\documentclass[12pt]{amsart}
\usepackage{amssymb}
\usepackage{color}
\usepackage{graphicx}
\usepackage{float}
\usepackage[all,cmtip]{xy}
\usepackage{tikz}
\usetikzlibrary{matrix}
\usepackage{url}
\usepackage{subfig}
\usepackage{hyperref}
\usepackage{comment}
\usepackage{xcolor}
\usepackage{ mathrsfs }

\newtheorem{theorem}{Theorem}[section]
\newtheorem{lemma}[theorem]{Lemma}
\newtheorem{proposition}[theorem]{Proposition}
\newtheorem{corollary}[theorem]{Corollary}

\newtheorem{definition}[theorem]{Definition}

\newtheorem{remark}[theorem]{Remark}

\newtheorem{example}[theorem]{Example}

\def\SS{\mathbb{S}}
\def\CC{\mathbb{C}}

\def\NN{\mathbb{N}}

\def\RR{\mathbb{R}}
\def\DD{\mathbb{D}}

\def\Diff{\operatorname{Diff}}

\def\Hom{\operatorname{Hom}}

\def\calR{\mathcal R}
\def\op{\mathcal Op}

\def\FR{\text{F}\mathcal{R}}
\def\HR{\text{H}\mathcal{R}}

\def\Id{\operatorname{Id}}
\def\slot{\!\text{ - }}

\usepackage{hyperref} 
\hypersetup{
    colorlinks=true,
    linkcolor=blue,
    filecolor=magenta,      
    urlcolor=cyan,
    citecolor=blue,
}

\NewCommandCopy{\proofqedsymbol}{\qedsymbol}
\newcommand{\remarkqedsymbol}{\text{\large$\bowtie$}}

\AtBeginEnvironment{proof}{\renewcommand{\qedsymbol}{\proofqedsymbol}}

\AtBeginEnvironment{example}{%
  \pushQED{\qed}\renewcommand{\qedsymbol}{\exampleqedsymbol}%
}
\AtEndEnvironment{example}{\popQED\endexample}
\AtBeginEnvironment{remark}{%
  \pushQED{\qed}\renewcommand{\qedsymbol}{\remarkqedsymbol}%
}
\AtEndEnvironment{remark}{\popQED\endremark}

\begin{document}

\title[Local h-principles for holomorphic partial differential relations]{Local $H$--principles for holomorphic partial differential relations}

\date{\today}


\author{Luis Giraldo}
\address{Departamento de Álgebra, Geometría y Topología, Fac. Matemáticas and IMI, Universidad Complutense de Madrid, 
	Facultad de Matem\'{a}ticas.}
\email{luis.giraldo@mat.ucm.es}

\author{Guillermo Sánchez-Arellano}
\address{Departamento de Álgebra, Geometría y Topología, Fac. Matemáticas, Universidad Complutense de Madrid, 
	Facultad de Matem\'{a}ticas, and Instituto de Ciencias Matem\'{a}ticas CSIC-UAM-UC3M-UCM, C. Nicol\'{a}s Cabrera, 13-15, 28049 Madrid, Spain.}
\email{guillermo\_sanchez@ucm.es}

\begin{abstract}
We introduce the notion of the realifications of an arbitrary 
\emph{holomorphic partial differential relation} $\calR$, that are partial differential relations associated to the restrictions of $\calR$ to totally real submanifolds of maximal dimension. Our main result states that if any realification of an open holomorphic partial differential relation over a Stein manifold satisfies a relative to domain $h$--principle, then it is possible to deform any formal solution into one that is holonomic in a neighbourhood of a Lagrangian skeleton of the Stein manifold. If the Stein manifold is an open Riemann surface or it has finite type, then that skeleton is independent of the formal solution. This yields the existence of local $h$--principles over that skeleton. These results broaden those obtained by F. Forstneri\v{c} and M. Slapar on holomorphic immersions, submersions and complex contact structures for instance to holomorphic local $h$--principles for the corresponding version in the complex category of some other classical examples of distributions and structures in the smooth category such as complex even contact, complex Engel and complex twisted locally conformal symplectic structures.  
\end{abstract}

\maketitle
\setcounter{tocdepth}{2} 
\tableofcontents

\section{Introduction}
\label{Section: Introduction}

A key feature of the Differential Geometry of the last 70 years has been to find the limits where a geometric problem begins to be directed by the rules of Differential Topology. Though several results in the work of H. Whitney hinted in that direction (\cite{Whitney_Weak_Embedding_Theorem}, \cite{Whitney_Strong_Embedding_Theorem}
), the first spectacular application of this approach was in the work of J. Nash about isometric embeddings, in particular he was able to show that there were no obstructions (apart from the smooth ones, provided by the Whitney embedding theorems) to $C^1$--isometrically embed a Riemannian manifold onto standard $\RR^N$ \cite{Nash_c1_Isometric_Embeddings}. This was kind of unexpected and it is in sharp contrast with the obstructions found by Riemann for the case of $C^{\infty}$ embeddings.

The theory began to take a familiar shape with the work of S. Smale and his student M. Hirsch in which they manage to reduce the whole theory of immersions to a set of obstruction theoretic topological invariants \cite{Hirsch_Thesis}, \cite{Smale_Classification_of_ImmersionsI}, \cite{Smale_Classification_of_ImmersionsII}. Moreover, the techniques developed to study this problem were later on captured in a general method to compute the homotopy type of several spaces of solutions of differential relations: what is known in modern language as the holonomic Lemma \cite{EliashbergMishachev}.

At that point, it became clear that there was a general theory behind the set of particular problems that the school of S. Smale were solving \cite{Phillips_Submersions_of_open_manifolds}, \cite{Phillips_foliations_on_open_manifoldsI},\cite{Phillips_foliations_on_open_manifoldsII}, \cite{Phillips_Smooth_maps_transverse_to_a_foliation},  \cite{Phillips_Smooth_maps_of_constant_rank}. This was brilliantly captured by M. Gromov by introducing the {\em homotopy principle} that is general guiding light and can be expressed as: ``A differential relation" (subset of a jet bundle) 
is declared to  satisfy the $h$--principle if, assuming that the space of geometric objects can be understood as the space of sections of a general fiber bundle $Y \to B$, and lifting the sections to the $r$--jets of sections, then any open\footnote{It is possible to drop that condition in some cases} set $\calR$ of the bundle $\calR \subset Y^{(r)} \to B$ is said to satisfy the $h$--principle if the whole space of sections of this bundle is homotopically equivalent to the space of {\em holonomic solutions}, i.e. sections of the bundle  $Y \to B$  whose associated lifted $r$--jets live in $\calR$. In other words, we say that the relation of order $r$ satisfies an $h$--principle if the set of solutions of the geometric problem is homotopically equivalent to the set of formal solutions, i.e. formal $r$-jets in which we do not care about making sure that the derivatives are coupled.

What Gromov proposed was to find as many instances as possible of relations satisfying that principle. Gromov's insight was really deep since the principle has been satisfied by a huge number of relations.

The usual procedure is as follows: isolate a class of Differential Relations, and then prove that all of them fulfill the $h$-principle. The aim of this work is to generalize this method to differential relations in holomorphic jet spaces of holomorphic fiber bundles.

One step in this direction was recently given by Forstneri\v{c} in \cite{Forstneric2020}. There it is proved that given a 
formal holomorphic contact form on a Stein manifold $X$, there exists a homotopy of formal holomorphic contact forms that starts on the original solution and terminates on a form that is a holonomic contact form inside a Stein domain $\Omega\subset X$ diffeotopic to $X$. Other similar results where proved by Forstneri\v{c} and Slappar in \cite{ForstericSlapar} for holomorphic immersions and submersions. Those results follow the following steps: 
\begin{enumerate}
    \item Formulate a holomorphic partial differential relation.
    \item Restrict to a totally real submanifold and check if you can find $h$--principles with classical methods (like convex integration).
    \item Extend the $r$--jet found over the totally real submanifold to a holomorphic $r$--jet using the Cauchy-Riemann condition to fix the holomorphic $r$--jet in the normal direction.
    \item Approximate it by a section that is holomorphic in a tubular neighborhood of the totally real submanifold.
    \item Do the previous steps inductively in each cell of a Lagrangian skeleton of your former Stein manifold.
\end{enumerate}

In this article we follow those steps to broaden Forstneri\v{c}'s and Slappar's results to general differential relations in holomorphic jet spaces of holomorphic vector bundles (Theorem \ref{Theorem: Theorem for arbitrary type Stein manifolds. Version that preserves holonomy in proper sets.}) or even in jet spaces of general holomorphic fiber bundles (Remark \ref{Remark: theorems work for general fiber bundles.}). We obtain our results introducing 
the realifications (Definition \ref{Definition: Realification}). That allows to obtain homotopies for a more general type of relations, not just the ones studied using convex integration. More precisely, if those realifications satisfy an $h$--principle obtained by any technique, then we obtain automatically an $h$--principle for the original relation. 

Given a holomorphic partial differential relation $\calR$ in the jet bundle of holomorphic sections of a holomorphic bundle $X\to B$, the realifications of $\calR$ lie in the jet bundle of smooth sections of the restriction of $X$ to a totally real submanifold of maximal dimension $M\subset B$. We use those realifications as a bridge from the holomorphic setting to the smooth one, so that we can make use of well known techniques to find $h$--principles. In order to go back to the holomorphic category we now need to approximate the smooth homotopies of formal solutions over $M$ by homotopies of holomorphic formal solutions defined over an open neighbourhood of $M$.

In order to do so we have adapted a Mergelyan approximation Theorem for the case of parametric sections of vector bundles (Theorem \ref{theorem: Mergelyan}). We use that Theorem to obtain homotopies satisfying the desired properties in a Stein domain containing the totally real submanifold.
This allows to proceed inductively over the descending disks of a Morse type strongly plurisubharmonic function that conform a Lagrangian skeleton of $B$ to prove Theorems \ref{Theorem: Theorem for arbitrary type Stein manifolds. Version that preserves holonomy in proper sets.} and \ref{Theorem: Theorem for arbitrary type Stein manifolds. Version that preserves holomorphy in proper sets.}.

At the end of this process we obtain a homotopy of formal solutions that finishes in a holonomic solution over a Stein neighbourhood of a Lagrangian skeleton that is diffeotopic to $B$. A priori, this skeleton depends on the initial formal data, but if the Stein manifold is an open Riemann surface or it has finite type (i.e. the strongly plurisubharmonic function defining the skeleton has a finite number of critical points), then that skeleton is independent of the formal solution. This yields the existence of local $h$--principles over that skeleton. This is the content of Theorems \ref{Theorem:  Theorem for finite type Stein manifolds. Version that preserves holonomy in proper sets.}, \ref{Theorem: Theorem for finite type Stein manifolds. Version that preserves holomorphy over the skeleton and closeness over proper sets.} and \ref{Theorem: Theorem for Riemann surfaces}.

To end this introduction, let us state the main yoga of this note: any differential relation that satisfies an $h$--principle over a manifold $M$ satisfies an $h$--principle for the ``complexified'' relation over a small neighborhood of the zero section of the cotangent bundle $T^*M$. We prove actually more, just understanding that a cotangent bundle is a specific type of Stein manifold and the zero section is somehow generalized to a Lagrangian skeleton of that manifold.

The content of this article is organized as follows. In Section \ref{Section: Holomorphic partial differential relations and their h-principles}, after making a short review of some of the $h$--principles that are known in the smooth category, we define what is a holomorphic partial differential relation (the analogous of a partial differential relation in the holomorphic setting) and the types of $h$--principles that their germs of formal solutions can satisfy. This settles down the general framework and the terminology that will be used along the manuscript. 

After that, in Section \ref{Section: Statements of the Theorems} we proceed to state our Theorems that will be proved in Section \ref{Section: Proofs of the Theorems} using the tools introduced in Section \ref{Section: Holomorphic approximation} such as the \emph{properly attached} subsets and the proof of the parametric version of a Mergelyan approximation Theorem for vector bundles. 

Finally, in Section \ref{Section: Applications.} we provide new proofs of some results that are already known (maps of fiberwise maximal rank and germs of complex Contact structures) and give some new $h$--principles for germs of complex Engel and complex twisted locally conformal symplectic structures.

The content of this article is part of the PhD thesis of the second name author \cite{TesisGuille}.

\textbf{Acknowledgments}
The authors are supported by the research program PID2021-126124NB-I00, the second author is also supported by the research program PID 2019-108936GB-C21 and by the 
project PR27/21-029.
 the second author was also partially supported by Beca de Personal Investigador en Formación UCM. They want to thank Álvaro del Pino for suggesting this problem and Yasha Eliashberg, Fabio Gironella, Fracisco Javier Martínez-Aguinaga, Eduardo Fernández and Francisco Presas for their help and fruitful suggestions. They also want to thank the referees for their careful reading and sugestions that had helped to improve and clarify the article. The second author is also specially grateful to Eduardo Fernández and Yasha Eliashberg for their invitations to visit the Departments of Mathematics of University of Georgia Athens and Stanford respectively during the Spring of 2023.

\section{Holomorphic partial differential relations and their \texorpdfstring{$h$}{h}--principles}\label{Section: Holomorphic partial differential relations and their h-principles}

\subsection{Review of the smooth case.}\hfill

\begin{definition}
A \emph{partial differential relation} $\calR$ is a subset of the $r$--jet space $X^{(r)}$ of a fiber bundle $\pi:X\to B$ for some $r\in\NN$. A section $\sigma:B\to X^{(r)}$ is a \emph{formal solution} of $\calR$ if $\sigma(B)\subset\calR$. A formal solution is \emph{holonomic} if it is the $r$-jet of a section $\sigma^{(0)}:B\to X$.

A differential relation \emph{satisfies an $h$--principle} if for every formal solution $\sigma_0$ there exists a homotopy of formal solutions $\sigma_t, t\in[0,1]$ that joins $\sigma_0$ with a holonomic solution $\sigma_1$.

Denote by $\FR$ the space of formal solutions of $\calR$ and by $\HR$ the space of holonomic solutions. We say that a differential relation \emph{satisfies a full $h$--principle}, or that the $h$--principle is \emph{complete}, if the inclusion $\HR\hookrightarrow\FR$ is a weak homotopy equivalence.
\end{definition}

One can also define some special types of $h$--principles such as \emph{relative to domain, parametric, relative to parameter} or \emph{$C^0$--dense} $h$--principles (see section 6.2 in \cite{EliashbergMishachev}).

Imposing some conditions on the differential relation or on the bundle $X\to B$ one can warrant the existence of $h$--principles.

Recall that a subset $\Omega$ of an affine space $A$ is called \emph{ample} if either $\Omega=\emptyset$ or if the convex hull of each path-connected component of $\Omega$ is $A$.
The fiber $(\pi^1_0)^{-1}(x), x\in X,$ of the projection $\pi^1_0:X^{(1)}\to X$ can be identified with $\Hom(T_pB,Vert_x)$, where $p=\pi(x)$ and $Vert_x:=\ker d\pi_{|T_xX}$ is the vertical subbundle, i.e. the subspace of $T_xX$ that is tangent to the fiber $X_p:=\pi^{-1}(p)$. \emph{Principal subspaces} are those affine subspaces of $\Hom(T_pB,Vert_x)$ conformed by all the morphisms that,
given a hyperplane $H_p$ of $T_pB$, extend the same given morphism $\eta:H_p\to Vert_x$.

\begin{definition}
A differential relation is \emph{open} if it is open as a subset of $X^{(r)}$. 

A differential relation $\calR\subset X^{(1)}$ is \emph{ample in principal directions} if $\calR$ intersects all principal subspaces of each fiber of $\pi^1_0:X^{(1)}\to X^{(0)}=X$ along ample sets.
\end{definition}

Using convex integration one can find 
$h$--principles for differential relations that are both open and ample \cite{EliashbergMishachev}. 
In addition to Gromov's convex integration \cite{Gromov_convex_integration}, \cite{Gromov} other important techniques that are useful to find $h$--principles have been developed . The interested reader can find some examples among the following list.
\begin{itemize}
    \item \emph{Removal of singularities} \cite{GromovEliashberg_removal_of_singularities}.
    \item \emph{Holonomic approximation} \cite{EliashbergMishachev}, that is a version of other Gromov's technique known as \emph{continuous sheaves} (or \emph{covering homotopy}) \cite{Gromov}.
    \item Lohkamp's theory of negative Ricci curvature \cite{Lohkamp_curvature_hp}.
    \item Asymptotically holomorphic theory \cite{Donaldson_divisor}.
    \item Existence and classification of overtwisted contact structures in all dimensions \cite{BEM}.
    \item Loose Engel structures \cite{Loose_Engel}.
\end{itemize}

\subsection{Holomorphic partial differential relations}\hfill

Fix a holomorphic fiber bundle $X\to B$. We will consider 
its jet spaces of holomorphic sections $X_\CC^{(r)}$ as subbundles of $X^{(r)}$. A  \emph{holomorphic partial differential relation} is a partial differential relation $\calR\subset X_\CC^{(r)}$. Cauchy-Riemann equations give that $X_\CC^{(r)}\subset X^{(r)}$ is a closed submanifold of positive codimension. Hence, there is no open differential relation $\calR\subset X^{(r)}$ that lies into $X_\CC^{(r)}$, so by an open holomorphic partial differential relation $\calR$ will always mean $\calR$ open in $X^{(r)}_\CC$.

Given a point $x\in X$, we identify the fiber $(\pi^1_{0|X^{(1)}_\CC})^{-1}(x)$ with the complex linear morphisms $\Hom_\CC(T_pB,Vert_x)$, where $p=\pi(x)$. \emph{Principal complex subspaces} are those affine subspaces of $\Hom_\CC(T_pB,Vert_x)$ conformed by all the morphisms that, given a complex hyperplane $H_p\subset T_pB$, extend the same given complex linear map \mbox{$\eta:H_p\to Vert_x$}.

Recall that a totally real submanifold $M$ of a complex manifold $B$ with complex structure $J$ is a submanifold such that $T_pM\cap JT_pM=\{0\}$, for every $p\in M$. Therefore its maximal (real) dimension is equal to the complex dimension of $B$. Given a totally real submanifold of maximal dimension $M\subset B$ we define the restriction map $\rho_M$ that sends each section $s:\op(p)\to X$ to $s_{|\op(p)\cap M}:\op(p)\cap M\to X_{|M}$.

The map $\rho_M$ induces the isomorphisms of bundles $\rho_M^r:(X_\CC^{(r)})_{|M}\longrightarrow(X_{|M})^{(r)}, r\in \NN\cup\{0\}$, that sends each $r$--tangency class of a holomorphic section $s$ over $p\in M$ to the $r$--tangency class of $\rho_M(s)$ at $p$. Take an $r$--tangency class in $(X_{|M})^{(r)}$ and $\zeta$ an analytic representative of it, then $(\rho_M^r)^{-1}$ sends the taken class to the one represented by the complex analytic extension of $\zeta$.

\begin{definition}\label{Definition: Realification}
Let $X\to B$ be a holomorphic fiber bundle and $\mathcal{R} \subset X_\CC^{(r)}$ be a holomorphic partial differential relation. For any given smooth totally real submanifold $M\subset B$ of maximal dimension, \emph{the induced realified relation} is
$$
\calR_\RR(M):=\rho_M^r(\calR)\subset(X_{|M})^{(r)}.
$$
\end{definition}

\begin{remark}\label{Remark: Isomorphism between realifications and complexifications}
    Note that $\rho^1_M:(X^{(1)}_\CC)_{|M}\to (X_{|M})^{(1)}$ is an isomorphism of complex affine bundles. Indeed, let $A\in (X^{(1)}_{\CC})_{|M}$. The fiber at the point $x:=\pi^1_0(A)$ of the affine bundle $\pi^1_0:X^{(1)}_{\CC}\to X$  can be identified with the vector space over $\CC$ formed by all the complex linear maps from $T_pB$ to $Vert_x$, where $A\equiv0$. Analogously, the fiber over $\rho^1_M(A)$ of  $(X_{|M})^{(1)}\to X_{|M}$ can be identified with the complex vector space of real linear maps from $T_pM$ to $Vert_x$ where $\rho^1_M(A)=0$. In this setting, the map $\rho^1_M$ sends each linear map to its restriction to $T_pM$ and its inverse is just the $\CC$--linear extension from $T_pM$ to $T_pB\equiv T_pM\oplus i T_pM$. 
    This shows in particular that $\rho^1_M$ sends principal complex subspaces of $(X^{(1)}_\CC)_{|M}$ to principal subspaces of $(X_{|M})^{(1)}$.
    Actually $\rho^r_M$ is an isomorphism of complex affine bundles between $(X^{(r)}_\CC)_{|M}\to (X^{(r-1)}_\CC)_{|M}$ and $(X_{|M})^{(r)}\to(X_{|M})^{(r-1)}$ for any $r\in \NN$, but we will not need it in this paper.
\end{remark}

The previous remark is the key to understanding how to use the techniques provided by the smooth theory in the holomorphic setting. With those techniques we can obtain homotopies of formal solutions of the realifications. Then we use the isomorphism described above to understand them as homotopies of formal solutions of the restricted holomorphic partial differential relation. Hence, we need to guarantee that the realifications satisfy an $h$--principle. We provide a family of examples where this works in the rest of this subsection.

\begin{definition}\label{Def. Open/Ample holomorphic relation}
Let $X\to B$ be a holomorphic fiber bundle and $\calR\subset X_\CC^{(r)}$. We will say that $\calR$ \textit{has open realifications} if for any totally real submanifold $M\subset B$ of maximal dimension, we have that $\calR_\RR(M)\subset (X_{|M})^{(r)}$ is open. Similarly, if $r=1$ we say that $\calR$ \textit{has ample realifications} if its realifications are ample in  
principal directions.
\end{definition}

Note that any open holomorphic relation,
has open realifications.
One important family of holomorphic relations with ample realifications are thick holomorphic relations.

\begin{definition}\label{Definition: THR}
    Let $X\to B$ be a holomorphic fiber bundle and $\calR\subset X_\CC^{(1)}$ a holomorphic partial differential relation on it. We say that $\calR$ is a \textit{thin holomorphic relation} (or a \textit{thin holomorphic singularity}) if it intersects every principal complex subspace of $X_\CC^{(1)}$, in a complex analytic set of complex codimension greater or equal than 1. We will say that $\calR\subset X_\CC^{(r)}$ is a \textit{thick holomorphic relation} (THR) if it is the complement of a thin holomorphic relation.
\end{definition}

Recall the definition of complex analytic set 
(p.86 in \cite{GunningRossi}). A subset $S$ of a complex manifold $X$ is complex analytic if there exists a covering by holomorphic charts, $\phi_j:B_j\to U_j$, where $B_j\subset X, U_j=B(0,\varepsilon)\subset\CC^n$ for $\varepsilon>0$ small enough, such that for any $B_j$ there exists a finite set of holomorphic functions $f_1,\ldots,f_m$ defined over $U_j$, satisfying that $\phi_j(B_j\cap S)=\{z\in U_j: f_1(z)= f_2(z)=\ldots=f_m(z)=0\}$.
It is known that a complex analytic set $S$ can have several irreducible components. The codimension of $S$ is the minimum of the codimensions of its irreducible components.

\begin{proposition}\label{proposition:THR implies ample realifications}
Let $\pi:X\to B$ be a holomorphic vector bundle and $\calR\subset X_\CC^{(1)}$ be a holomorphic partial differential relation. If $\calR$ is a THR then it has ample realifications.
\end{proposition}
\begin{proof}
Let $M\subset B$ be a totally real submanifold of maximal dimension. We have to prove that $\calR_\RR(M)$ is ample in $(X_{|M})^{(1)}$.

By Remark \ref{Remark: Isomorphism between realifications and complexifications} the restriction map $\rho^1_M$ sends principal complex subspaces of $(X^{(1)}_\CC)_{|M}$ to principal subspaces of $(X_{|M})^{(1)}$. Therefore the complement of $\calR_\RR(M),\Sigma,$ will intersect each principal subspace in a stratified subset of codimension $\ge 2$. Then $\Sigma$ is a thin singularity in the sense of \cite{EliashbergMishachev} and this yields that $\calR_\RR(M)$ is ample in principal directions.
\end{proof}

Note that Proposition \ref{proposition:THR implies ample realifications} works because complex hypersurfaces do not disconnect the ambient space, so their complements are ample. This is not true in the real analytic category. Indeed, the complement of a real analytic set of codimension 1 can be formed by two convex connected components, so it will not be ample. Since there are many geometric objects that can be understood as nonzero sections of vector bundles, this observation will lead us to some differences between the real and complex cases. They will be explored in Section \ref{Section: Applications.}.

\begin{remark}\label{Remark: THR implies conditions of theTheorem}
Note that if $\calR$ is a THR, then every type of $h$--principle applies for its realifications. Indeed, since it is open, the relation has open realifications. Therefore, since the realifications are open and ample we can apply Gromov's convex integration \cite{Gromov} to them. We can take advantage of this fact in order to prove $h$--principles near totally real submanifolds for the relation $\calR$ (or even more general kind of sets).
More precisely, we will prove some $h$--principles for germs of holomorphic sections over adapted skeletons of Stein manifolds of finite type (see Definition \ref{Definition: Adapted Skeleton} and Theorem \ref{Theorem: Theorem for finite type Stein manifolds. Version that preserves holonomy in proper sets.}).
\end{remark}

\subsection{Types of \texorpdfstring{$h$}{h}--principles for germs of sections of holomorphic partial differential relations}\hfill

We now adapt the different types of $h$--principles described in section 6.2 of \cite{EliashbergMishachev} for the case of germs of sections of holomorphic partial differential relations. For those readers that are already familiar with the concepts of relative to domain, parametric, relative to parameter and $C^0$--dense $h$--principles we recommend to focus in Definitions \ref{Definition: weakly relative to domain h-principle}, \ref{Definition: Holomorphic h-principle} and \ref{Definition: Holomorphic weakly relative h-principle} and use the rest of the subsection just to check 
the notation that will be used along the article.

\begin{definition}
Let $S\subset B$ be a closed subset. Given a fiber bundle $X\to B$, A \emph{germ of section over} $S$ is the following set of data:
\begin{itemize}
    \item A choice of a basis $\mathcal{B}$ of open neighbourhoods of $S$,
    \item A choice of sections $\sigma_{U} \in Sect(U, X)$ for each element $U$ of $\mathcal{B}$,
    \item A compatibility condition that is written as follows. Denote the restriction map of two open sets $U_\alpha, U_\beta\in\mathcal{B}$ such that $U_\alpha\subset U_\beta$ by 
    $r_{\beta}^\alpha: Sect(U_\beta, X) \to Sect(U_\alpha,X)$. Then $r_{\beta}^\alpha( \sigma_\beta) =\sigma_\alpha$.
\end{itemize}
Two germs of sections are equivalent if in the union of their chosen bases, the choices of sections define a unique germ.
\end{definition}

In practice, a germ can be defined as a choice of a section over any open neighbourhood of $S$ and two germs are equivalent if there is an open set inside the intersection of the domains where the sections coincide.

\begin{definition}
A germ of a section of $X^{(r)}$ over $S$ is \emph{holonomic} if there exists an open neighbourhood of $S$ where it coincides with the $r$--jet of a section of $X$ over $\op(S)$.
\end{definition} 

Following Gromov (see \cite{Gromov} p.85), we denote by $\op(A)$ an arbitrarily small but non-specified open neighbourhood of a subset $A$ that may be reduced if it is necessary for the argument. Most of the times $\op(A)$ works as a replacement of the sentence \textit{an open neighbourhood of $A$}.

\begin{definition}
Fix a partial differential relation $\calR\subset X^{(r)}$ and the fibration $\pi_\calR: \calR \to B.$ The \emph{set of germs of formal extended solutions}, $\widetilde{F\calR}(S)$,  is the set of germs of sections of the fibration $\pi_\calR$ over $S$. The \emph{set of germs of holonomic extended solutions} $\widetilde{H\calR}(S)$, is the set of germs of formal extended solutions that are holonomic in some neighbourhood of $S$.

A partial differential relation $\calR\subset X^{(r)}$ \emph{satisfies an $h$--principle over germs of a closed subset} $S\subset B$ if for every germ of formal extended solution $\sigma_0$ there exists a homotopy of germs of formal extended solutions $\sigma_t, t\in[0,1]$ that joins $\sigma_0$ with a germ of holonomic extended solution $\sigma_1$. 
We say that $\calR$ \textit{satisfies a full $h$--principle over germs of $S$} if the inclusion $i:\widetilde{\HR}(S)\hookrightarrow\widetilde{\FR}(S)$ is a weak homotopy equivalence.
\end{definition}

\begin{definition}
Let $S\subset B$ be a closed subset of $B$ and $C$ a closed subset of $S$. Let $\sigma:\op(S)\to X^{(r)}$ be in $\widetilde{F\calR}(S,C):=\widetilde{F\calR}(S)\cap\widetilde{H\calR}(C)$, where $\calR\subset X^{(r)}$ is a partial differential relation of the bundle $X\to B$. We will denote the space of germs of formal extended solutions over $S$ that coincide with $\sigma_{|C}$ at $C$ by $\widetilde{F\calR}(S,C,\sigma)$. We define $\widetilde{H\calR}(S,C,\sigma)$ analogously.

The $h$--principle over germs of $S$ is \textit{relative to $C$} if for every $\sigma_0\in\widetilde{\FR}(S,C)$ there exists a homotopy $\sigma_t\in\widetilde{\FR}(S,C,\sigma_0), t\in[0,1]$, such that $\sigma_1\in\widetilde{\HR}(S,C,\sigma_0)$. The $h$--principle over germs is \textit{relative to closed domain} if it is relative to every closed $C\subset S$. 
\end{definition}

Since a holomorphic section is determined by its restriction to an open set, finding $h$-principles that are relative to domains with nonempty interiors is impossible for most of the differential relations that lie into $X_\CC^{(r)}$ for some holomorphic vector bundle $X\to B$. That is the reason why we give a weaker notion of relativity to domain.

\begin{definition}\label{Definition: weakly relative to domain h-principle}
Let $\widetilde{\FR}(S,C,\sigma,\varepsilon)$ be the space of germs of formal of extended solutions over $S$ that are in $\widetilde{\HR}(C)$ and are 
$\varepsilon$--close to $\sigma_{|C}$ over $C$ in $C^0(C,X^{(r)}_{|C})$. Let $\widetilde{\HR}(S,C,\sigma,\varepsilon)$ be the analogous space for germs of holonomic extended solutions.

The $h$--principle over germs of $S$ is \textit{weakly relative to $C$} if for every $\varepsilon>0$ and for every $\sigma_0\in\widetilde{F\calR}(S,C)$ there exists a homotopy $\sigma_t\in\widetilde{\FR}(S,C,\sigma_0,\varepsilon), t \in[0,1]$ such that $\sigma_1\in\widetilde{\HR}(S,C,\sigma_0,\varepsilon)$ for every $\varepsilon>0$ small enough.
The $h$--principle is \emph{weakly relative to domain} if it is weakly relative to every closed domain $C\subset S$.
\end{definition}

Now let us consider fibered differential relations $\calR\subset P\times X^{(r)}$. They can be understood 
as a family of partial differential relations $\calR_p=(\{p\}\times X^{(r)})\cap\calR\hookrightarrow X^{(r)}$ that are continuously dependent on some parameter $p$ in a compact 
space $P$, that we will assume to admit a CW-complex structure throughout this article.

Let the set of \emph{fibered germs of formal extended solutions} be $$\widetilde{F\calR}(S):=\{\sigma:P\times\op(S)\to\calR |\sigma_p=\sigma(p,\text{-})\in\widetilde{\FR}_p(S), \forall p\in P\}.$$

The sets $\widetilde{F\calR}(S,C), \widetilde{F\calR}(S,C,\sigma)$ and $\widetilde{F\calR}(S,C,\sigma,\varepsilon)$ are defined analogously, and also their corresponding sets of parametric germs of holonomic extended solutions for any closed subset $C\subset S$, any $\sigma\in\widetilde{H\calR}(C)$ and any $\varepsilon>0$.

All the previous types of $h$--principles for germs can also be defined for fibered differential relations. In this case we will also say that the $h$--principles for germs are \emph{relative to $Q$}, a compact 
 subset 
(that will always be assumed to be a CW-subcomplex) of $P$, if the homotopies can be chosen to be fixed at $Q$ when the former parametric germ of extended formal solution is holonomic over $Q\times\op(S)$.  
We also say that the $h$--principle is relative to the fiber
if it is relative to every compact subset $Q\subset P$.

\begin{definition}
Let $\calR\subset X^{(r)}$ be a partial differential relation of $X\to B$, let $S\subset B$ be a closed subset, let $P$ be a compact 
space and let $Q\subset P$ be a compact subset. Then $\calR$ satisfies an $h$--principle that is parametric with parameter $P$ relative to $Q$ if the fibered relation $\calR_P:=P\times\calR$ satisfies an $h$--principle that is relative to $Q$. In addition
\begin{itemize}
    \item
    the parametric over $P$ $h$--principle is relative to parameter if it is relative to $Q$ for every compact subset $Q\subset P$, 
    \item
    the $h$--principle is parametric if it is parametric for every compact 
    space $P$,
    \item
    the $h$--principle is parametric and relative to parameter if it is parametric and for every compact 
    space $P$, the parametric over $P$ $h$--principle is relative to parameter.
\end{itemize}
We define analogously the full, parametric and/or relative to parameter versions of the relative and weakly relative to domain $h$--principles for germs. Note that it is also possible for fibered relations.
\end{definition}

\begin{remark}\label{Remark: Parametric and relative to parameter h-principle implies full h-principle}
Note that an $h$--principle is full if and only if it is parametric and relative to parameter.

Indeed, the $h$--principle is parametric and relative to parameter if and only if it is parametric for each disk $\DD^k$ and relative to its boundary $\partial\DD^k$ for each $k\in\NN\cup\{0\}$. This is equivalent to say that the relative homotopy groups $\pi_k(\widetilde{\FR}(S),\widetilde{\HR}(S))=0$ for every $k$. Then, the long exact sequence
\begin{multline*}
\ldots
\longrightarrow
\pi_k(\widetilde{\HR}(S))
\overset{i_*}{\longrightarrow}
\pi_k(\widetilde{\FR}(S))
\longrightarrow
\pi_k(\widetilde{\FR}(S),\widetilde{\HR}(S))
\longrightarrow
\pi_{k-1}(\widetilde{\HR}(S))
\longrightarrow\\
\ldots
\longrightarrow
\pi_0(\widetilde{\FR}(S),\widetilde{\HR}(S))
\end{multline*}
yields that $\pi_k(\widetilde{\FR}(S),\widetilde{\HR}(S))=0$, for every $k\in\NN\cup\{0\}$ if and only if $\pi_k(\widetilde{\HR}(S))\overset{i_*}{\cong}\pi_k(\widetilde{\FR}(S))$ for every $k\in\NN\cup\{0\}$.
\end{remark}

Now we give the notion of $C^0$--density in the context of $h$--principles over germs of sections.

\begin{definition} 
An $h$--principle over germs of a closed subset $S$ is \textit{$C^0$--dense} if for any germ of formal extended solution $\sigma$ and every $U=\op(\sigma^{(0)}(S)) \subset X$, where $\sigma^{(0)}$ stands for the 0-jet part of $\sigma$, (i.e. $\sigma^{(0)}:=\pi^r_0\circ\sigma$) there exists a homotopy $\sigma_t$, $t\in[0,1]$, joining $\sigma=\sigma_0$ with a germ of holonomic extended solution $\sigma_1$, so that $\sigma^{(0)}_t(S)\subset U$, for all $t\in[0,1]$.

We can define the $C^0$--dense versions of all the previously defined $h$--principles.
\end{definition}

To end this subsection we define one last type of $h$--principle for germs of extended solutions of holomorphic partial differential relations.

\begin{definition}\label{Definition: Holomorphic h-principle}
Let $\calR\subset X^{(r)}_\CC$ be a holomorphic partial differential relation of a holomorphic bundle $X\to B$ and let $S\subset B$ be a closed subset. Denote by $\mathcal{O}\widetilde{\FR}(S)$ the subset of $\widetilde{\FR}(S)$ conformed by those germs of extended formal solutions that are holomorphic in $\op(S)$. The $h$--principle over germs is \emph{holomorphic} if for every  $\sigma_0\in\mathcal{O}\widetilde{\FR}(S)$ there exists a homotopy $\sigma_t\in\mathcal{O}\widetilde{\FR}(S)$ such that $\sigma_1\in\widetilde{\HR}(S)$. Analogously we can define the full, parametric, relative to parameter and $C^0$--dense \emph{holomorphic} $h$--principles.
\end{definition}

Let $\calR\subset X_\CC^{(r)}, r\ge 1$ be a holomorphic partial differential relation. Note that by the identity principle, since $X_\CC^{(r)}$ is a holomorphic bundle, then every holomorphic section is uniquely determined by its restriction to an open subset. Therefore, if $S$ is connected and if $C\subset S$, then $\widetilde{H\calR}(S)=\mathcal{O}\widetilde{F\calR}(S)\cap\widetilde{H\calR}(C)$. Therefore the definition of weakly relative to domain given above is not useful in this context, indeed it attains for homotopies that preserve holonomy over $C$. Nevertheless sometimes we may be interested in finding homotopies made of holomorphic germs of extended formal solutions over $S$ that are close to a germ of extended holonomic solution over $C$. This yields the following definition of weakly relative to domain in the context of holomorphic $h$--principles.

\begin{definition}\label{Definition: Holomorphic weakly relative h-principle}
Let $\calR\subset X_\CC^{(r)}$ be a holomorphic partial differential relation of a holomorphic bundle $X\to B$, let $C\subset S\subset B$ be closed subsets and let $\sigma\in\widetilde{F\calR}(C)$ a germ of extended formal solution.
We will say that $\sigma$ is a \emph{pseudo-holonomic germ of extended solution over $C$} if $j^r(\sigma^{(0)})\in\widetilde{\HR}(C)$ and the linear interpolation between $\sigma$ and $j^r(\sigma^{(0)})$ is formed by formal solutions of $\calR_{|\op(C)}$, where $\sigma^{(0)}$ represents the projection to the $0$--jet space of the section $\sigma$. 

Let us denote by $\widetilde{p\HR}(C)$ the set of pseudo-holonomic germs of extended solutions over $C$, and let $\mathcal{O}\widetilde{\FR}(S,C):=\mathcal{O}\widetilde{\FR}(S)\cap\widetilde{p\HR}(C)$.
Let $\sigma\in\mathcal{O}\widetilde{\FR}(S,C)$ and let $\mathcal{O}\widetilde{F\calR}(S,C,\sigma,\varepsilon)$ be the subset of sections of $\mathcal{O}\widetilde{F\calR}(S,C)$ that are $\varepsilon$--close to $j^r(\sigma^{(0)})_{|C}$ over $C$ in $C^0(C,X^{(r)}_{|C})$.

The \emph{holomorphic $h$--principle of $\calR$ is weakly relative to $C$} if for every $\varepsilon>0$ and for every $\sigma_0\in\mathcal{O}\widetilde{F\calR}(S,C)$ there exists a homotopy $\sigma_t\in\mathcal{O}\widetilde{F\calR}(S,C,\sigma_0,\varepsilon), t\in[0,1]$ such that $\sigma_1\in\widetilde{H\calR}(S)$.\emph{The holomorphic $h$--principle is weakly relative to domain} if it is weakly relative to every closed domain $C\subset S$. 

We define similarly the \emph{full, $C^0$--dense, parametric and/or relative to parameter weakly relative to domain holomorphic $h$--principles}. 
\end{definition}

\section{Statements of the Theorems}\label{Section: Statements of the Theorems}

\subsection{\texorpdfstring{$H$}{H}--principles for germs over skeletons of finite type.}\hfill

The following Theorem is the main result of this article and it will be proven in section \ref{Section: Proofs of the Theorems}. 

\begin{theorem}\label{Theorem:  Theorem for finite type Stein manifolds. Version that preserves holonomy in proper sets.}
Let $B$ be a Stein manifold of finite type and $X\to B$ a holomorphic vector bundle. Let $\calR$ be a holomorphic partial differential relation that is open in $X^{(r)}_\CC$. If $\calR_\RR(M)$ satisfies a relative to domain $h$--principle for every totally real submanifold of maximal dimension $M\subset B$, then $\calR$ satisfies the $h$--principle over germs of any adapted skeleton. The $h$--principle over germs is weakly relative to Stein compact domains properly attached to the adapted skeleton.

Moreover, if the $h$--principles that satisfy the realifications $\calR_\RR(M)$ are all parametric, relative to parameter, full and/or $C^0$--dense, then so it is the $h$--principle over germs of sections lifting to $\calR$.
\end{theorem}

There are several notions in the previous statement that we have not defined yet. Their precise definitions will be provided in section \ref{Section: Holomorphic approximation}.
One can think of an adapted skeleton (see definition \ref{Definition: Adapted Skeleton}) as the Lagrangian skeleton of a strongly plurisubharmonic exhaustive function. They are not exactly the same, but they share a lot of properties. Moreover, given a Lagrangian skeleton one can find an adapted skeleton arbitrarily $C^0$--close to it. The proof of Theorem \ref{Theorem:  Theorem for finite type Stein manifolds. Version that preserves holonomy in proper sets.} goes inductively over the ``bones" of the skeleton. The key additional property that satisfy the adapted skeletons is precisely that their ``bones" are \emph{properly attached} (see definition \ref{Definition: Properly attached sets}). This is needed to use a Mergelyan Approximation Theorem in each inductive step to obtain holonomic germs of extended solutions from holonomic solutions over the realifications.

Assuming additional conditions we can give the following result for general holomorphic fiber bundles.

\begin{corollary}\label{Corollary:theorem for fiber bundles.}
Let $B$ be a Stein manifold of finite type and $X\to B$ a holomorphic fiber bundle. Consider an open holomorphic partial differential relation $\calR\subset X^{(r)}_\CC$ and a holomorphic section $s:B\to X$.  
Let $N_s$ be a tubular neighbourhood of $s(B)$ in $X$ and $\calR_s=\calR\cap (\pi^r_0)^{-1}(N_s)$, where $\pi^r_0:X^{(r)}\to X$ is the natural projection. If $\calR_\RR(M)$ satisfies a $C^0$--dense and relative to domain $h$--principle for every totally real submanifold of maximal dimension $M\subset B$, then $\calR_s$ satisfies the $h$--principle for germs of sections over any adapted skeleton. The $h$--principle over germs is $C^0$--dense and weakly relative to Stein compact domains properly attached to the adapted skeleton.

Moreover, if the $h$--principles that satisfy the realifications $\calR_\RR(M)$ are all parametric, relative to parameter and/or full, then so it is the $h$--principle over germs of sections lifting to $\calR_s$.
\end{corollary}
\begin{proof}[Proof (of Corollary \ref{Corollary:theorem for fiber bundles.})]
$N_s$ is a holomorphic vector bundle over $B$. We just need to check that the realifications of $\calR_s$ satisfy the conditions of Theorem \ref{Theorem: Theorem for finite type Stein manifolds. Version that preserves holonomy in proper sets.} to obtain the result. We need to check that the realifications of $\calR_s$ satisfy the same type of $h$--principle than the realifications of $\calR$.

This is given by the $C^0$--density of the $h$--principles over the realifications, indeed the homotopies from $\FR_{s,\RR}(M)$ to $\HR_{s,\RR}(M)$ can be chosen to have their images inside $N_s^{(r)}=(\pi_0^r)^{-1}(N_s)$ for every totally real submanifold $M\subset B$ of maximal dimension.
\end{proof}

\begin{remark}\label{Remark: theorems work for general fiber bundles.}
All the results in this section are stated for the case in which $X\to B$ is a holomorphic vector bundle but, similarly to Theorem \ref{Theorem: Theorem for finite type Stein manifolds. Version that preserves holonomy in proper sets.} and Corollary \ref{Corollary:theorem for fiber bundles.}, all of them work for general holomorphic bundles assuming that the $h$--principles of the realifications are $C^0$--dense and that the germs of formal extended solutions have images of their $0$--jet parts inside a tubular neighbourhood of a holomorphic section.
\end{remark}

As a direct consequence of Theorem \ref{Theorem: Theorem for finite type Stein manifolds. Version that preserves holonomy in proper sets.} we have the following result for open holomorphic partial differential relations with ample realifications. In particular, by Proposition \ref{proposition:THR implies ample realifications}, it works for any THR.

\begin{theorem}\label{Theorem: Holomorphic convex integration}
Let $B$ be a Stein manifold of finite type and $X\to B$ be a holomorphic vector bundle. Then every open holomorphic relation $\calR\subset X^{(1)}_\CC$ with ample realifications satisfies a full, $C^0$--dense, parametric and relative to parameter $h$--principle over germs of any adapted skeleton that is weakly relative to any Stein compact properly attached to the adapted skeleton.
\end{theorem}
\begin{proof}[Proof (of Theorem \ref{Theorem: Holomorphic convex integration}).]
Since $\calR_\RR(M)$ is open and ample for every totally real submanifold $M$ of maximal dimension, Gromov's convex integration gives that $\calR_\RR(M)$ satisfies all types of $h$--principles. Therefore we are under the strongest hypothesis of Theorem \ref{Theorem: Theorem for finite type Stein manifolds. Version that preserves holonomy in proper sets.}.
\end{proof}

The homotopies of germs of formal solutions obtained in Theorem \ref{Theorem: Theorem for finite type Stein manifolds. Version that preserves holonomy in proper sets.} are made to preserve the holonomy over properly attached Stein compact sets, but they are not formed by holomorphic germs of sections. We can do the opposite and obtain the following

\begin{theorem}\label{Theorem: Theorem for finite type Stein manifolds. Version that preserves holomorphy over the skeleton and closeness over proper sets.}
Under the hypotheses of Theorem \ref{Theorem: Theorem for finite type Stein manifolds. Version that preserves holonomy in proper sets.}, the relation $\calR$ satisfies the holomorphic $h$--principle over germs of any adapted skeleton. This holomorphic $h$--principle is weakly relative to Stein compact domains properly attached to the adapted skeleton. 

Moreover if the $h$--principles satisfied by the realifications are all parametric, relative to parameter, full and/or $C^0$--dense, then so it is the holomorphic $h$--principle of germs over the adapted skeleton.
\end{theorem}

\begin{remark}\label{Remark: theorems work for holomorphic h-principles.}
Likewise Theorem \ref{Theorem: Theorem for finite type Stein manifolds. Version that preserves holonomy in proper sets.}, all of its consequences stated above have analogue results for holomorphic $h$--principles over germs of sections, i.e. preserving holomorphy instead of holonomy over properly attached sets.
\end{remark}

\subsection{Theorems for Stein manifolds of arbitrary type}\hfill

If we do not
assume that the base manifold $B$ is of finite type,we are not able to fix the adapted skeleton and therefore we cannot obtain an $h$--principle for germs over it for complex dimensions $>1$.
Nevertheless, we can prove the following Theorems.

\begin{theorem}\label{Theorem: Theorem for arbitrary type Stein manifolds. Version that preserves holonomy in proper sets.}
Let $B$ be a Stein manifold with complex structure $J$ and let $X\to B$ be a holomorphic vector bundle with a holomorphic partial differential relation $\calR$ that is open in $X^{(r)}_\CC$. Assume that $\calR_\RR(M)$ satisfies a relative to domain $h$--principle for every totally real submanifold of maximal dimension $M\subset B$.
Then, for every formal solution $\sigma_0:B\to\calR$ there exists a smooth family of diffeomorphisms $h_t:B\to h_t(B)\subset B, t\in [0,1]$ and a homotopy of formal solutions $\sigma_t:B\to\calR$ such that
\begin{enumerate}
    \item $h_0=\Id_B$
    \item $(h_t(B),J_{|h_t(B)})$ is Stein (or equivalently, $(B,h_t^*J)$ is Stein) for every $t\in[0,1]$.
    \item $\sigma_1$ is holonomic in $h_1(B)$.
\end{enumerate}
Moreover, if the $h$--principles that satisfy the realifications $\calR_\RR(M)$ are all $C^0$--dense, then the homotopy of formal solutions can be chosen to satisfy that $\sigma_t^{(0)}$ is arbitrarily close to $\sigma_0^{(0)}$. If the $h$--principles over the realifications are parametric (and relative to parameter) then the previous theorem holds for continuous families of formal solutions (relatively to closed sets of the parameter space) with $h_t$ being independent of the parameter.
\end{theorem}
Let $X\to B$ be a holomorphic fiber bundle over a Stein manifold $B$ with complex structure $J$ and let $\calR\subset X^{(r)}_\CC$ be a holomorphic partial differential relation. Let us consider the space 
$$
\mathscr{D}(B,J)
\equiv
\{h:B\overset{\cong}{\to}h(B)\subset B\, \text{diffeomorphism}|(B,h^*J)\,\text{is Stein}\}
$$
and let $\mathscr{D}_0(B,J)$ be the connected component of the identity of $\mathscr{D}(B,J)$. Let us consider also the spaces of pairs
$$
\mathscr{F}\mathcal{R}
\equiv
\{
(\sigma,h)|h\in\mathscr{D}_0(B,J)\,\text{and}\,\sigma\in\operatorname{F}\!h^*\calR
\}
$$
and
$$
\mathscr{H}\mathcal{R}
\equiv
\{
(\sigma,h)|h\in\mathscr{D}_0(B,J)\,\text{and}\,\sigma\in\operatorname{H}\!h^*\calR
\}.
$$
Finally, let us call $\mathscr{F}\calR_h$ and $\mathscr{H}\calR_h$ the preimage of the diffeomorphism $h\in\mathscr{D}_0(B,J)$ under the natural projection over the second component of $\mathscr{F}\mathcal{R}$ and $\mathscr{H}\mathcal{R}$ respectively.

Using that notation, the previous Theorem can be stated in the following way

\begin{corollary}
Let $B$ be a Stein manifold with complex structure $J$ and let $X\to B$ be a holomorphic vector bundle with a holomorphic partial differential relation $\calR$ that is open in $X^{(r)}_\CC$ and assume that $\calR_\RR(M)$ satisfies a relative to domain $h$--principle for every totally real submanifold of maximal dimension $M\subset B$. Then, for every $\sigma\in\FR$ there exists a path $(\sigma_t,h_t)\in\mathscr{F}\calR, t\in[0,1]$ such that $(\sigma_0,h_0)=(\sigma,\Id)$ and $(\sigma_1,h_1)\in\mathscr{H}\calR$. 
Moreover, if the $h$--principles that satisfy the realifications are $C^0$--dense, then the homotopy can be chosen to satisfy that $(\sigma_t\circ h_t^{-1})^{(0)}$ is $C^0$--close to $\sigma$ for every $t\in[0,1]$.

If the $h$--principles over the realifications are parametric (and relative to parameter), then for every compact 
space $P$ (every closed subset $Q\subset P$), and for every  $s:P\to\mathscr{F}\calR_{\Id}$ (such that $s(Q)\subset\mathscr{H}\calR_{\Id}$) there exists a path $h_t\in\mathscr{D}_0(B,J)$ and a path of maps $s_t:P\to\mathscr{F}\calR_{h_t}$ such that ($s_{t|Q}\equiv h_t^*s_{|Q}$ for every $t\in[0,1]$) $s_0=s$ and $s_1(P)\subset\mathscr{H}\calR_{h_1}$. 
\end{corollary}

\begin{theorem}\label{Theorem: Theorem for arbitrary type Stein manifolds. Version that preserves holomorphy in proper sets.}
Let $B, J, X$ and $\calR$ be as in Theorem \ref{Theorem: Theorem for arbitrary type Stein manifolds. Version that preserves holonomy in proper sets.}. Then, for every holomorphic formal solution $\sigma_0:B\to\calR$ there exists a smooth family of diffeomorphisms $h_t:B\to h_t(B)\subset B, t\in[0,1]$ and a homotopy of holomorphic formal solutions $\sigma_t:h_1(B)\to\calR_{|h_1(B)}$ such that
\begin{enumerate}
    \item 
    $h_0=\Id_B$,
    \item
    $(h_t(B), J_{|h_t(B)})$ is Stein (or equivalently, $(B,h_t^*J)$ is Stein) for every $t\in[0,1]$,
    \item
    $\sigma_1$ is holonomic in $h_1(B)$
\end{enumerate}
Moreover, if the $h$--principles that satisfy the realifications $\calR_\RR(M)$ are all $C^0$--dense, then the homotopy of formal solutions can be chosen to satisfy that $\sigma_t^{(0)}$ is arbitrarily $C^0$--close to $\sigma_0^{(0)}$. If the $h$--principles over the realifications are parametric (and relative to parameter), then the previous theorem holds for continuous families of holomorphic formal  solutions (relatively to closed sets of the parameter space) with $h_t$ being independent of the parameter.
\end{theorem}

\subsection{The case of Riemann surfaces}\hfill

The fact that every open subset of an open Riemann surface is Stein makes special the case of complex dimension 1. Indeed, Remark \ref{Remark: We have to perturb or substitute the Lagrangian skeleton} yields that in this case we will not have to ask the Stein manifold to have finite type to obtain $h$--principles for germs over any Lagrangian skeleton. Moreover, it follows 
from the definitions (see \ref{Definition: Stein compact and Holomorphically convex compact set} and \ref{Definition: Properly attached sets}) that this fact also implies that any compact subset of an open Riemann surface is Stein compact and properly attached to any other compact subset. Since in addition the totally real submanifolds of maximal dimension of a Riemann surface are precisely the smooth curves on it we will obtain the following
\begin{theorem}\label{Theorem: Theorem for Riemann surfaces}
   Let $\Sigma$ be an open Riemann surface and let $X\to \Sigma$ be a holomorphic vector bundle with a holomorphic partial differential relation $\calR$ that is open in $X_\CC^{(r)}$. If $\calR_\RR(\Gamma)$ satisfies a relative to domain $h$--principle for every smooth curve $\Gamma\subset\Sigma$, then $\calR$ satisfies the $h$--principle over germs of any Lagrangian skeleton. The $h$--principle is weakly relative to any compact domain.

   Moreover, if the $h$--principles that satisfy the realifications $\calR_\RR(\Gamma)$ are all parametric, relative to parameter, full and/or $C^0$--dense, then so it is the $h$--principle over germs of sections lifting to $\calR$.
\end{theorem}
\begin{remark}\label{Remark: Riemann surfaces results}
    We can obtain analogues of Corollary \ref{Corollary:theorem for fiber bundles.}, Remark \ref{Remark: theorems work for general fiber bundles.} and Theorem \ref{Theorem: Holomorphic convex integration}. Moreover, as in Remark \ref{Remark: theorems work for holomorphic h-principles.}, we can also get analogues for holomorphic $h$--principles preserving holomorphy instead of holonomy along compact subsets.
\end{remark}

\section{Holomorphic approximation}
\label{Section: Holomorphic approximation}

One key step in the proofs of the Theorems consists on approximating smooth germs of extended solutions over totally real submanifolds by holomorphic ones. This can be done for functions near totally real submanifolds thanks to a Mergelyan type result (Theorem 20 in \cite{HolomorphicApproximation}). Moreover, if the function is already holomorphic in an open neighbourhood of a \emph{Stein compact} set $K$, the approximation can be done to be arbitrarily close to the original over $K$, provided that $K$ is \emph{properly attached} to the totally real submanifold (see Definitions \ref{Definition: Stein compact and Holomorphically convex compact set} and \ref{Definition: Properly attached sets}). 

Theorem \ref{theorem: Mergelyan} is a generalization of that Mergelyan type Theorem. This new version allows us to approximate sections of holomorphic vector bundles parametrically instead of just functions.

\subsection{Adapted skeletons and properly attached sets}\label{Subsection: Properly attached sets}\hfill

\begin{definition} \label{Definition: Totally real stratification by disks}
Let $B$ be a complex manifold and $M\subset B$ be a subset. $M$ is a \emph{totally real stratification by affine strata} if $M$ is presented as a countable disjoint union $M=\bigcup_{i=1}^k M_i, k\in\NN\cup\{\infty\}$, such that each $M_i$ is a totally real submanifold of $B$ diffeomorphic to a $k_i$--dimensional open ball. We denote by $M_{\le i}$ the totally real stratification by affine strata $M_{\le i}:=\bigcup_{j=1}^i M_j$. 
\end{definition}

\begin{example}\label{Example and definition: Lagrangian skeleton}
Let $B$ be a Stein manifold with complex structure $J$ and let $\phi$ be an exhausting strongly plurisubharmonic Morse function on $B$. The 2-form 
$$
\omega_\phi:=-dd^\CC\phi:=-d(d\phi \circ  J)
$$
is a symplectic structure compatible with $J$. Consider $g_\phi$ its associated Riemannian metric and $X_\phi:=\nabla_{g_\phi}\phi$ the gradient of $\phi$ for the metric $g_\phi$.
Assume that the set of critical points of $\phi$, $\operatorname{Crit}(\phi)=\{p_i\}_{i=1}^k, k\in\NN\cup\{\infty\}$, satisfy that $\phi(p_i)<\phi(p_{i+1})$ for each $i<k$. 
The stable disks $M_i$ of the critical points of $\phi$ 
for the flow of $X_\phi$ are isotropic submanifolds (see Lemma 2.21 in \cite{EliCie}), therefore they define a totally real stratification by affine strata $M_\phi$. We will call that stratification \emph{the Lagrangian skeleton of $\phi$}. 
\end{example}

\begin{definition} \label{Definition: Stein compact and Holomorphically convex compact set}
Let $K\subset B$ be a compact subset of a complex manifold $B$. 
\begin{itemize}
    \item We say that $K$ is \emph{Stein compact} if it admits a basis of Stein neighbourhoods.
    \item We say that $K$ is \emph{holomorphically convex} if it admits an open Stein neighbourhood $\Omega\subset B$ such that $K$ is $\mathcal{O}(\Omega)$--convex. 
\end{itemize}
\end{definition}
A compact set $K$ in a complex space $\Omega$ is \emph{$\mathcal{O}(\Omega)$--convex} if $K=\widehat{K}_{\mathcal{O}(\Omega)}$, where $\mathcal{O}(\Omega)$ is the space of holomorphic functions in $\Omega$ and
$$\widehat{K}_{\mathcal{O}(\Omega)}:=\{z\in\Omega:|f(z)|\le\max_{x\in K}\{|f(x)|\},\forall f\in\mathcal{O}(\Omega) \}$$
is the \emph{$\mathcal{O}(\Omega)$--convex hull of $K$}.

By Proposition 2.5.5 in \cite{SteinMflds}, if $K$ is holomorphically convex, then it is automatically Stein compact with a basis of Stein neighbourhoods
$$
\Omega_c=\{z\in B:\phi(z)<c\},
c\in(0,\infty)
$$
for some plurisubharmonic function $\phi:\op(K)\to[0,\infty)$ that is strongly plurisubharmonic outside of $K$ such that $K=\phi^{-1}(0)$ is $\mathcal{O}(\Omega_c)$--convex for every $c\in(0,\infty)$.

\begin{definition}\label{Definition: Properly attached sets}
Let $K$ be a compact subset of a complex manifold $B$ and let $M\subset B$ be a totally real submanifold. We say that \emph{$K$ and $M$ are properly attached} if $K\cup M$ is Stein compact.

If instead $M=\bigcup_{i=1}^k M_i, k\in\NN\cup\{\infty\}$, is a totally real stratification by affine strata, \emph{$K$ is properly attached to $M$} if $K\cup M_{\le i}$ is properly attached to $M_{i+1}$ for every $i<k$.
\end{definition}
\begin{remark}
\begin{itemize}
    \item Note that, by Lemma 2
    in \cite{HolomorphicApproximation}, if $M$ is a totally real submanifold properly attached to a compact subset $K$, then $K$ is automatically also Stein compact.
    \item Let 
    $K$ be a holomorphically convex compact domain with smooth boundary. Let $\Sigma:=\partial K$, the field of complex tangencies of $\Sigma$, $\xi:=T\Sigma\cap J(T\Sigma)$, is a contact distribution on $\Sigma$. Let $M$ be a totally real submanifold transverse to $\Sigma$ such that $\Lambda:=M\cap \Sigma$ is an integral submanifold of $\xi$, i.e.  $\Lambda$ satisfies that 
    $T\Lambda\subseteq \xi$. Then, by Corollary 8.26 and Section 10.2 of \cite{EliCie}, if $S\cup M$ is compact, then $K$ is properly attached to $M$. 
\end{itemize}
\end{remark}
\begin{example}\label{Example: Lagrangian skeleton seems to be good}
Let $B,\phi$ and $M_\phi=\bigcup_{i=1}^k M_i, k\in\NN\cup\{\infty\}$ as in Example \ref{Example and definition: Lagrangian skeleton}. Let $p_i\in\operatorname{Crit}(\phi)$ and let $c\in\RR$ be a regular value such that $\phi(p_{i-1})<c<\phi(p_i)$.
Let $K$ be the sublevel set of $c$ for $\phi$, i.e. $K:=\{x\in B:\phi(x)\le c\}$. Since $K$ is compact and holomorphically convex and $M_i\cap\partial K$ is isotropic for $\omega_\phi$ (\cite{EliCie} Theorem 5.7 and Remark 2.23.(b)), then $M_\phi$ is properly attached to $K$.

Indeed, Theorem 8.5 in \cite{EliCie} shows that, given $U\subset B$ an open neighbourhood of $K\cup M_i$, there exist a strongly plurisubharmonic exhausting function $\phi':B\to\RR$ with $\operatorname{Crit}(\phi)=\operatorname{Crit}(\phi')$ that coincides with $\phi$ in $\op(K)$ and outside of $\op(M_i)$ and such that there exists a regular value $c'$ for $\phi'$ such that the sublevel set $K':=\{x\in B:\phi'(x)\le c'\}$ satisfies that $K\cup M_i\subset K'\subset U$.
\end{example}
\begin{remark}\label{Remark: We have to perturb or substitute the Lagrangian skeleton}
Following with the notation of the previous example, since every open subset of an open Riemann surface is Stein (see \cite{SteinMflds}), if $\dim_\CC B=1$ then $M_{i+1}$ is properly attached to the holomorphically convex set $K'$. In higher dimensions this may not be true, but the previous example shows that each descending disk of $M_{\phi'}$ is properly attached to $K'$.
\end{remark}
To prove several of the Theorems of this article we need to apply Theorem \ref{theorem: Mergelyan} successively over the ``bones'' of the skeleton. But as it is shown in Remark \ref{Remark: We have to perturb or substitute the Lagrangian skeleton}, we cannot do that in general to a Lagrangian skeleton without perturbing or substituting their disks in each step. The following type of skeleton shares many properties with Lagrangian ones and we will be able to perform that inductive process to them with their disks remaining fixed.
\begin{definition}\label{Definition: Adapted Skeleton}
Let $B$ be a Stein manifold and let $\phi:B\to \RR$ be a plurisubharmonic Morse exhausting function. Enumerate the critical points $p_1,p_2,\ldots,p_i,p_{i+1},\ldots$ of $\phi$ to satisfy that $\phi(p_i)<\phi(p_{i+1})$ and denote by $k_i$ the Morse index of $p_i$. A totally real stratification by affine strata $M=\bigcup_{i=1}^k M_i, k\in\NN\cup\{\infty\}$ is a \emph{skeleton adapted to $\phi$} or an \emph{adapted skeleton} if the following holds:
\begin{enumerate}
    \item\label{Adapted skeleton property 1} $\phi$ does not have critical points in $B\setminus M$.
    \item\label{Adapted skeleton property 2} Each disk $M_i$ contains the critical point $p_i$ and $\dim M_i=k_i$.
    \item\label{Adapted skeleton property 3} There exists an isotopy $h_t:B\to B,t\in[0,\infty)$ such that $h_0=\Id, h_{t|M}=\Id$ for all $t\in[0,\infty)$, $M=\bigcap_{t\in[0,\infty)}h_t(B)$ and
    each regular sublevel set $\{\phi\circ h_t^{-1}<c\}$ is Stein.
    \item\label{Adapted skeleton property 4} Each $M_{\le i}$ is a Stein compact set.
\end{enumerate}
\end{definition}
Theorem 8.32 in \cite{EliCie} states that an affine stratification by disks satisfying conditions \ref{Adapted skeleton property 1}, \ref{Adapted skeleton property 2} and \ref{Adapted skeleton property 3} in the previous Definition always exists. The proof of that Theorem shows 
that the last condition also holds\footnote{ In the notation of \cite{EliCie}, the sets bounded by the convex hypersurfaces $\Sigma^{(r)}_i$ form the basis of Stein neighbourhoods of each $M_{\le i}$.}. Note that this last condition in Definition \ref{Definition: Adapted Skeleton} gives that if $M=\bigcup_{i=1}^k M_i, k\in\NN\cup\{\infty\}$, is an adapted skeleton, then for every $i\le k$, $M_{\le i}$ is properly attached to $M$.

\subsection{Parametric Mergelyan Theorem for sections of holomorphic vector bundles}
\label{Subsection: Parametric Mergelyan Theorem for holomorphic vector bundles}
\hfill

The following Mergelyan approximation Theorem is proven in \cite{HolomorphicApproximation}, Theorem 20, for functions (i.e. the case where $X=B\times\CC$) in the non-parametric case. Following similar arguments to the proof of the Theorem 4.3 in
\cite{parametricMergelyan} 
we can prove Lemma \ref{Lemma: Parametric Mergelyan for functions}, that is a parametric and relative to parameter version for functions. After that we can use the arguments at the end of the proof of Theorem 2.8.4 in \cite{SteinMflds} to give the following parametric Mergelyan approximation Theorem for sections of a holomorphic vector bundle.

\begin{theorem}\label{theorem: Mergelyan}
Let $X\to B$ be a holomorphic vector bundle over a complex manifold $B$. Let $K\subset B$ be a Stein compact subset and $M\subset B$ be an embedded totally real closed $C^r$--submanifold (possibly with boundary) such that $S=K\cup M$ is also Stein compact in $B$. Then for any $C^r$ section $\sigma:S\to X$ that is holomorphic over $\op(K)$ there exists a sequence of sections $\sigma_k$ that are holomorphic over $\op(S)$ such that
\begin{equation*}
    \lim_{k\to\infty}\parallel \sigma_k-\sigma\parallel_{C^r(S)}=0.
\end{equation*}
This Theorem also holds parametrically for compact parameter spaces and relative to parameter. 

More precisely, given a family of $C^r$ sections $\sigma_p:S\to X$ that are holomorphic over $\op(K)$ depending continuously on a parameter $p$ in a compact space $P$, there exists a sequence of continuous families of sections $\sigma_{p,k}$ that are holomorphic over $\op(S)$ such that
\begin{equation*}
    \lim_{k\to\infty}\parallel \sigma_{p,k}-\sigma_p\parallel_{C^r(S)}=0.
\end{equation*}
Moreover, if there is a compact subset $Q\subset P$ such that $\sigma_q$ is already holomorphic over $\op(S)$ for every $q\in Q$, then the sequence can be chosen to satisfy that $\sigma_{q,k}=\sigma_q$ for every $q\in Q,k\in\NN$.
\end{theorem}
The following Lemma is the parametric case for functions. It will be useful to prove Theorem \ref{theorem: Mergelyan}.

\begin{lemma}\label{Lemma: Parametric Mergelyan for functions}
Let $K$ and $S=K\cup M$ be Stein compacts in a complex manifold $B$, where $M=\overline{S\setminus K}$ is an embedded totally real submanifold (possibly with boundary) of class $C^r$. Then for any family $f_p:S\to \CC$ of $C^r$ functions that are holomorphic over $\op(K)$ and that depends continuously on $p\in P$, there exists a sequence of continuously depending families of holomorphic functions $\tilde{f}_{p,k}:\op(S)\to\CC$ such that
\begin{equation*}
    \lim_{k\to\infty}\parallel \tilde{f}_{p,k}-f_p\parallel_{C^r(S)}=0.
\end{equation*}
If in addition there is a compact subset $Q\subset P$ such that $f_q$ is already holomorphic over $\op(S)$ for every $q\in Q$, then the sequence of families can be chosen to satisfy that $\tilde{f}_{q,k}=f_q$ for every $q\in Q$.
\end{lemma}
\begin{proof}[Proof (of Lemma \ref{Lemma: Parametric Mergelyan for functions}).]
 Let $\varepsilon$ be a positive number. Now select a finite set of points $\{p_1,\ldots,p_m\}\subset P$ and a covering by open sets $\{P_j\}_{j=1}^m$, such that $p_j\in P_j$ for every $j=1,\ldots,m$ and that
 \begin{equation*}
     \parallel f_p-f_{p_j}\parallel_{C^r(S)}<\frac{\varepsilon}{4},\quad \text{for every } p\in P_j, j=1,\ldots,m.
 \end{equation*}
 Now, use the nonparametric version of this theorem for the case of functions (Theorem 20 in \cite{HolomorphicApproximation}) to obtain holomorphic functions $g_j:\op(S)\to\CC$ such that
 \begin{equation*}
     \parallel g_j-f_{p_j}\parallel_{C^r(S)}<\frac{\varepsilon}{4},\quad j=1,\ldots,m.
 \end{equation*}
 Then take a continuous partition of unity $\chi_j$ subordinated to $\{P_j\}_{j=1}^m$ and define
 \begin{equation*}
     \tilde{f}_p:=\sum_{j=1}^m\chi_j(p)g_j, \quad\text{for }p\in P.
 \end{equation*}
 It is easy to check that $\tilde{f}_p$ is holomorphic for every $p\in P$, moreover we have that
 \begin{equation*}
     \parallel \tilde{f}_p-f_p\parallel_{C^r(S)} \le\sum_{j=1}^m\chi_j(p)\parallel g_j-f_p\parallel_{C^r(S)}<\frac{\varepsilon}{2},
 \end{equation*}
since $\parallel g_j-f_p\parallel\le\parallel g_j-f_{p_j}\parallel+\parallel f_{p_j}-f_p\parallel<\frac{\varepsilon}{2}$ for $p\in P_j$ and $\chi_j(p)=0$ for $p\not\in P_j$. This completes the proof for the nonrelative case.
 
Now, to make $\tilde{f}_p$ coincide with $f_p$ at $p\in Q$ consider $N\subset B$ a compact neighbourhood of $S$ such that $f_q$ are holomorphic at the interior of $N$ for every $q\in Q$.
The space of such functions is a Banach space, so Michael's extension Theorem (see Theorem 2.8.2 in \cite{SteinMflds} for the precise statement, that is proven in \cite{Michael_extension_theorem}) 
 guarantees that there exists a continuous family $\xi_p, p\in P,$ of functions that are holomorphic in the interior of $N$ such that $\xi_q=f_q$ for every $q\in Q$ and let $P_0$ be a neighbourhood of $Q$ such that $\parallel \xi_p-f_p\parallel_{C^r(S)}<\frac{\varepsilon}{2}$ for all $p\in P_0$. Now let $\rho:P\to[0,1]$ be a continuous function supported on $P_0$ such that $\rho_{|Q}\equiv 1$ and define
 \begin{equation*}
     \tilde{f}'_p:=\rho(p)\xi_p+(1-\rho(p))\tilde{f}_p,\quad\text{for }p\in P,
 \end{equation*}
 that is a family of holomorphic functions continuously dependent on $p\in P$, with $\parallel \tilde{f}'_p-f_p\parallel_{C^r(S)}<\frac{\varepsilon}{2}$ and such that $\tilde{f}'_q=f_q$ for every $q\in Q$. The sequence $\tilde{f}_{p,k}$ required can be obtained doing the same process for $\varepsilon=\frac{1}{k}$.
\end{proof}

\begin{proof}[Proof (of Theorem \ref{theorem: Mergelyan}).]
Since $S$ is a Stein compact and the result is local, it is enough to prove it in the case in which $B$ is a Stein manifold.

As the bundle $X\to B$ is finitely generated, there exists a surjective holomorphic vector bundle map $f:B\times\CC^N\to X$ for some $N\in\NN$. We can embed $X$ as a holomorphic vector subbundle of $B\times\CC^N$ such that $B\times\CC^N=X\oplus \ker f$ (see Corollary 2.6.6 in \cite{SteinMflds} or Corollary 5.29 in \cite{EliCie}), hence there exists a holomorphic projection $\pi:B\times\CC^N\to X$ such that $\ker\pi=\ker f$. This allows us to {consider} 
sections from $B$ to $X$ as maps to $\CC^N$. Therefore we can apply Lemma \ref{Lemma: Parametric Mergelyan for functions} to each component and then project the resulting map to $X$ by $\pi$ to obtain the result.
\end{proof}
 
\begin{corollary}\label{Corollary: Mergelyan for adapted skeletons}
Let $X\to B$ be a holomorphic vector bundle over a Stein manifold $B$ of finite type. Let $M=\bigcup_{i=0}^k M_i, k\in\NN$, be an adapted skeleton of $B$ and let $K\subset B$ be a holomorphically convex compact subset properly attached to $M$. For every family of $C^r$--sections $\sigma_p:\op(K\cup M)\to X$ that are holomorphic over $\op(K)$ and that are continuously dependent on a parameter $p$ in a compact 
space $P$ and for every $\varepsilon>0$, there exists a continuous family of sections $\tilde{\sigma}_p$ that are holomorphic over $\op(K\cup M)$ such that
$$
\parallel\tilde{\sigma}_p-\sigma_p\parallel_{C^0(K\cup M)}<\varepsilon,\quad \forall p\in P.
$$
Moreover, if there is a compact subset $Q\subset P$ such that for every $q\in Q$, the section $\sigma_q$ is holomorphic over $\op(K\cup M)$, then $\tilde{\sigma}_p$ can be chosen to satisfy that $\tilde{\sigma}_q=\sigma_q$ for every $q\in Q$.
\end{corollary}
\begin{proof}
Apply Theorem \ref{theorem: Mergelyan} to $\sigma_p$ over $K\cup M_1$ to obtain a section $\sigma_{p,1}$ that is holomorphic and $C^r$--close to $\sigma_p$ over $\Omega_1$, a Stein neighbourhood of $K\cup M_1$. Let $\rho_1:B\to\RR$ be a smooth cutoff function supported on $\Omega_1$ and such that $\rho_1\equiv 1$ over $\op(K\cup M_1)$. Now apply Theorem \ref{theorem: Mergelyan} to $\tilde{\sigma}_{p,1}:=\rho_1\sigma_{p,1}+(1-\rho_1)\sigma_p$ over $(K\cup M_1)\cup M_2$ to obtain a section $\sigma_{p,2}$ that is holomorphic and $C^r$--close to $\tilde{\sigma}_{p,1}$ (and therefore $C^0$--close to $\sigma_p$) over $\Omega_2$, a Stein neighbourhood of $K\cup M_{\le 2}$. Cutoff the section again to obtain $\tilde{\sigma}_{p,2}:=\rho_2\tilde{\sigma}_{p,2}+(1-\rho_2)\tilde{\sigma}_{p,1}$, where $\rho_2:B\to \RR$ is supported on $\Omega_2$ and is equal to 1 over $\op(K\cup M_{\le 2})$, and repeat the process until getting the last section $\tilde{\sigma}_p:=\tilde{\sigma}_{p,k}$.
\end{proof}
\begin{remark}\label{Remark: C^r-closeness of Mergelyan for subsets of adapted skeletons}
Although the $C^r$--closeness over $K\cup M$ cannot be achieved in general due to the Cauchy Riemann conditions, the approximation $\tilde{\sigma}_p$ obtained in the previous proof is $C^r$--close to $\sigma_p$ over $K$ and over each $M_{i+1}\setminus\op(M_{\le i})$.
\end{remark}

The next Corollary is just a version of Corollary 8.39 in \cite{EliCie} for parametric sections of vector bundles and noticing that there are places where we can obtain $C^r$--closeness instead of just $C^0$--closeness. Its proof is exactly the same but using our version of the Mergelyan Theorem instead of the one that works just for functions.

\begin{corollary}\label{Corollary: Mergelyan for infinite type manifolds}
Let $X\to B$ be a holomorphic vector bundle over a Stein manifold $B$ with complex structure $J$. Let $\phi:B\to\RR$ be a strongly plurisubharmonic Morse exhausting function. Let $K=\{x\in B:\phi(x)\le c\}$ be a sublevel set of $\phi$ for some regular value $c\in\RR$.

Then, for every family of $C^r$--sections $\sigma_p:B\to X$ that are holomorphic over $\op(K)$ and continuously dependent on a parameter $p$ in a compact 
space $P$ and for every positive function $\varepsilon:B\to\RR$, 
there exists an isotopy $h_t:B\hookrightarrow B, t\in[0,1]$ such that $h_0=\Id$ and such that $K\subset h_t(B)$ and $(h_t(B), J_{|h_t(B)})$ is Stein for every $t\in[0,1]$ satisfying that there exists a continuous family of sections $\tilde{\sigma}_p$ that are holomorphic over $h_1(B)$ such that $\tilde{\sigma}_p$ is arbitrarily $C^r$--close to $\sigma_p$ over $K$ for every $p\in P$ and such that
$$
|\tilde{\sigma}_p(x)-\sigma_p(x)|<\varepsilon(x),\quad \forall(p,x)\in P\times h_1(B).
$$
Moreover, if there is a compact subset $Q\subset P$ such that for every $q\in Q$, the section $\sigma_q$ is already holomorphic, then $\tilde{\sigma}_p$ can be chosen to satisfy that $\tilde{\sigma}_q=\sigma_q$ for every $q\in Q$. 
\end{corollary}
\begin{proof}
For simplicity of notation let us assume that there is only one critical point for each critical value. Let $p_1\in B$ be the only critical point satisfying that there is no critical value in the interval $(c,\phi(p_1))$. 
Take an increasing sequence of regular values of $\phi$, $c=c_0<c_1<c_2<\ldots$ such that $\phi(p_1)<c_1$ and such that there is only one critical value on each interval $(c_{j-1},c_{j})$, whose corresponding critical point is $p_j$ $j=1,\ldots,\infty$. Let $V_j:=\{x\in B:\phi(x)\le c_j\}, j=1,\ldots,\infty$ and let $M_1$ be the stable disk of $p_1$.

By Example \ref{Example: Lagrangian skeleton seems to be good}, $M_1$ is totally real and properly attached to $K$. Therefore we can use Theorem \ref{theorem: Mergelyan} to obtain an open neighbourhood $U_1=\op(K\cup M_1)\subset V_1$ where there is defined a holomorphic section $\sigma_1$ satisfying that 
$$
\parallel \sigma_1-\sigma\parallel_{C^r(\overline{U}_1)}<\frac{1}{2}\min_{ x\in \overline{U}_1}\{\varepsilon(x)\}.
$$
Extend $\sigma_1$ smoothly to the whole manifold satisfying that $|\sigma_1(x)-\sigma(x)|<\frac{1}{2}\varepsilon(x)$ for every $x\in B$ (maybe using a cutoff function and shrinking $U_1$ if necessary). Now use Theorem 8.5 in \cite{EliCie} to obtain a strongly plurisubharmonic Morse exhausting function $\phi_1$ such that
\begin{itemize}
    \item 
    $\phi_1=\phi$ outside $U_1$ and inside $\widetilde{V}_0=\op(K)$,
    \item 
    the only critical point of $\phi_{1|U_1\setminus\widetilde{V}_0
    }$ is $p_1$,
    \item 
    there exists a regular value $c'_1>\phi_1(p_1)$ such that $\widetilde{V}_0\cup M_1\subset V'_1:=\{x\in B:\phi_1(x)\le c'_1\}\subset U_1$.
\end{itemize}
Now let $M_2$ be the stable disk of $p_2$ for $\phi_1$. 
We can use Theorem \ref{theorem: Mergelyan} again to obtain an open neighbourhood $U_2=\op(V'_1\cup M_2)\subset V_2$ and (after extending smoothly to $B$) a section $\sigma_2:B\to X$ that is holomorphic in $U_2$, such that 
$$
\parallel \sigma_2-\sigma_1\parallel_{C^r(\overline{U}_2)}<\frac{1}{4}\min_{ x\in \overline{U}_2}\{\varepsilon(x)\}
$$
and such that
\begin{equation*}
   |\sigma_2(x)-\sigma_1(x)|<\frac{1}{4}\varepsilon(x) 
\end{equation*}
for every $x\in B$. Then use again Theorem 8.5 in \cite{EliCie} to obtain a strongly plurisubharmonic exhausting function $\phi_2$ such that
\begin{itemize}
    \item 
    $\phi_2=\phi_1$ outside $U_2$ and inside $\widetilde{V}_1=\op(V'_1),$
    \item
    the only critical point of $\phi_{2|U_2\setminus\widetilde{V}_1}$ is $p_2$,
    \item
    there exists a regular value $c'_2>\phi_2(p_2)$ such that $\widetilde{V}_1\cup M_2\subset V'_2:=\{x\in B:\phi_2(x)<c'_2\}\subset U_2$.
\end{itemize}

Now just repeat that process inductively to obtain a sequence of holomorphic sections $\sigma_i$, 
a sequence of strongly plurisubharmonic Morse exhausting functions $\phi_i$, 
a sequence of disks $M_i$ that are the stable manifolds of the points $p_i$ for $\phi_{i-1}$, 
a sequence $c'_i$ of regular values of $\phi_i$ such that $c'_i<\phi_i(p_i)$ 
and sequences of subsets $\widetilde{V}_{i-1}, V'_i$ and $U_i, i=1,\ldots$ such that
\begin{enumerate}
    \item\label{Contition 1 in Corollary: Mergelyan for infinite type manifolds }
    $\widetilde{V}_{i-1}\cup M_i\subset V'_i:=\{x\in B:\phi_i\le c'_i\}\subset U_i$,
    \item\label{Contition 2 in Corollary: Mergelyan for infinite type manifolds }
    $U_{i+1}=\op(V'_i\cup M_{i+1})\subset V_{i+1}$,
    \item\label{Contition 3 in Corollary: Mergelyan for infinite type manifolds }
    $\phi_{i}=\phi_{i-1}$ outside $U_{i}$ and inside $\widetilde{V}_{i-1}=\op(V'_{i-1})$,
    \item\label{Contition 4 in Corollary: Mergelyan for infinite type manifolds }
    the only critical point of $\phi_{i|U_{i}\setminus\widetilde{V}_{i-1}}$ is $p_i$,
    \item\label{Contition 5 in Corollary: Mergelyan for infinite type manifolds }
    $\sigma_i$ is holomorphic in $U_i$,
    \item\label{Contition 6 in Corollary: Mergelyan for infinite type manifolds }
    $\parallel\sigma_{i}-\sigma_{i-1}\parallel_{C^r(\overline{U}_i)}<\left(\frac{1}{2}\right)^i\min_{x\in\overline{U}_i}\{\varepsilon(x)\}$,
    \item\label{Contition 7 in Corollary: Mergelyan for infinite type manifolds }
    $|\sigma_i(x)-\sigma_{i-1}(x)|<\left(\frac{1}{2}\right)^i\varepsilon(x)$, for every $x\in B$,
\end{enumerate}
for each $i=1,\ldots$, where $\phi_0=\phi$ and $\sigma_0=\sigma$.

If $B$ is of finite type, the previous process finishes after $k\in\NN$ steps. At that moment we will have obtained a section $\sigma_k$ that is holomorphic over $U_k$. Now, by Theorems 5.7 and 5.18 in \cite{EliCie}, and the same argument that the one given in the proof of Theorem \ref{theorem: Mergelyan}, we can uniformly approximate $\sigma_k$ by a global holomorphic section $\tilde{\sigma}$ that satisfies the desired properties. In this case $h_t=\Id$ for all $t\in[0,1]$.

If $B$ is not of finite type, $\tilde{\sigma}$ is obtained as the limit of the sequence $\{\sigma_j\}_{j\in\NN}$ that is defined in $\tilde{B}:=\bigcup_{i\in\NN}\operatorname{Int}V'_i$.  Note that all the previous steps can be done fixed at $Q$ if the sections $\sigma_q$ are already holomorphic for every $q\in Q$.

The construction of the desired isotopy is as in the proof of Theorem 8.32(b) in \cite{EliCie}: Take regular values $c''_i>c_i$ of $\phi$ without critical values in $[c_i,c''_i]$, the sublevel set $V''_i:=\{x\in B:\phi(x)\le c''_i\}$ and $d_i:=\sum_{j=1}^i\frac{1}{2^j}$ for each $i\in\NN$. 

Now construct for each $i\in \NN$ a diffeotopy $h_t^i:B\to B, t\in[d_{i-1},d_{i}]$, where $d_0=0$ such that
\begin{itemize}
    \item 
    $h_t^i$ maps level sets of $\phi_k$ to level sets,
    \item
    $h_{d_{i-1}}^i=\Id$ and $h_{d_i}^i(V_i)=V'_i$,
    \item
    $h^i_t=\Id$ on $V_{i-1}$ and outside $V''_i$ for all $t\in[d_{i-1},d_{i}]$.
\end{itemize}
Now define the diffeotopies $h_t=h^i_t\circ h^{i-1}_{d_i}\circ\ldots\circ h^1_{\frac{1}{2}}$ for $t\in[d_{i-1},d_i]$ for each $i\in\NN$. Note that since $h_t$ stabilizes on compact sets there exists $h_1:=\lim_{t\to 1}h_t$ that is not surjective, therefore $h_t$ can be defined as an isotopy for $t\in[0,1]$.
\end{proof}

\begin{corollary}\label{Corollary: Mergelyan for Riemann surfaces}
    Let $X\to\Sigma$ be a holomorphic vector bundle over an open Riemann surface. 
    Let $M_\phi$ be a Lagrangian skeleton of $\Sigma$ for the strongly plurisubharmonic Morse exhaustion function $\phi$ and let $K\subset\Sigma$ be a compact subset. 
    For every family of $C^r$--sections $\sigma_p:\op(K\cup M_\phi)\to X$ that are holomorphic over $\op(K)$ and continuously dependent on a parameter $p$ in a compact 
    space $P$, and for every positive function $\varepsilon:K\cup M_\varphi\to\RR_+$, 
    there exists a continuous family of sections $\widetilde{\sigma}_p$ that are holomorphic over $\op(K\cup M_\phi)$ such that
    $$
    | \widetilde{\sigma}_p(x)-\sigma_p(x)|
    <\varepsilon(x),\quad\forall (p,x)\in P\times\op(K\cup M_\phi).
    $$
    Moreover, if there is a compact subset $Q\subset P$ such that for every $q\in Q$, the section $\sigma_q$ is holonomic over $\op(K\cup M_\phi)$, then $\widetilde{\sigma}_p$ can be chosen to satisfy that $\widetilde{\sigma}_q=\sigma_q$ for every $q\in Q$.
\end{corollary}
\begin{proof}
    The construction of $\widetilde{\sigma}$ follows the same steps as the ones in the proof of Corollary \ref{Corollary: Mergelyan for infinite type manifolds}. The only difference is that in the case of Riemann surfaces we can choose all the functions $\phi_i$ to be equal to $\phi$.
\end{proof}

\begin{remark}
Since every holomorphic fiber bundle with fiber biholomorphic to $\CC^n$ over a Stein manifold has a holomorphic section (Theorem 2.5.4 in \cite{Gromov_Affine_holomorphic_bundle_is_vector_bundle}), every holomorphic affine bundle is (non canonically) a holomorphic vector bundle. This means that all the results in this section apply also to affine bundles, in particular to holomorphic jet bundles.
\end{remark}

\section{Proofs of the Theorems}
\label{Section: Proofs of the Theorems}
The proof of Theorems \ref{Theorem: Theorem for finite type Stein manifolds. Version that preserves holonomy in proper sets.} and \ref{Theorem: Theorem for arbitrary type Stein manifolds. Version that preserves holonomy in proper sets.} follow the same type of argument.
The firs step is done by Lemma \ref{Lemma: Extending holonomy through totally real submanifolds}.
This Lemma uses the $h$--principles over the realifications together with Remark \ref{Remark: Isomorphism between realifications and complexifications} and Theorem \ref{theorem: Mergelyan} to construct homotopies of parametric germs of formal solutions over totally real manifolds with trivial complex normal bundles that begin in a prescribed parametric germ and that end in a holonomic one relatively to the fiber.
This allows to extend the holonomy from one stratum of the skeleton to the next one inductively. This inductive process will conclude the proof of Theorem \ref{Theorem: Theorem for arbitrary type Stein manifolds. Version that preserves holonomy in proper sets.}. To finish the proof of Theorem \ref{Theorem: Theorem for finite type Stein manifolds. Version that preserves holonomy in proper sets.} it only remains to choose the appropriate parametric spaces.

On the other hand, Theorems \ref{Theorem: Theorem for finite type Stein manifolds. Version that preserves holomorphy over the skeleton and closeness over proper sets.} and \ref{Theorem: Theorem for arbitrary type Stein manifolds. Version that preserves holomorphy in proper sets.} follow directly from the previous ones just by using Theorem \ref{theorem: Mergelyan}, Corollaries \ref{Corollary: Mergelyan for adapted skeletons} and \ref{Corollary: Mergelyan for infinite type manifolds} and Remark \ref{Remark: C^r-closeness of Mergelyan for subsets of adapted skeletons}.

\subsection{Extending holonomy and pseudo-holonomy through totally real submanifolds}
\label{Subsection: Extending holonomy and pseudo-holonomy through totally real submanifolds}\hfill

The following Lemma allows us to homotope germs of formal extended solutions that are holonomic in some Stein compact set $K$ to germs of extended solutions that are holonomic in $K$ and in a totally real manifold $M$ properly attached to $K$. It works parametrically and relative to parameter and it constitutes the key step of the proofs of Theorems \ref{Theorem: Theorem for finite type Stein manifolds. Version that preserves holonomy in proper sets.} and \ref{Theorem: Theorem for arbitrary type Stein manifolds. Version that preserves holonomy in proper sets.}.

\begin{lemma}\label{Lemma: Extending holonomy through totally real submanifolds}
Let $M\subset B$ be a totally real submanifold of class $C^r$ of a complex manifold $B$ with trivial complex normal bundle and let $K$ be a compact set properly attached to $M$.
Let $S:=K\cup M$, let $Q\subset P$ be a compact subset in a compact 
space and 
let $X\to B$ be a holomorphic vector bundle with a fibered open holomorphic relation $\calR\subset P\times X^{(r)}_\CC$ such that all its realifications satisfy a relative to domain and fiber $h$--principle. Let $\sigma\in\widetilde{\FR}(S)$ be a germ of formal extended solution that is already holonomic in $P\times\op(K)\cup Q\times\op(S)$. Then, for every $\varepsilon>0$ there exists $\Omega$ a Stein neighbourhood of $S$ and a homotopy $\sigma_t:P\times\Omega\to \calR$, $t\in[0,1]$ of germs of formal extended solutions such that
\begin{enumerate}
    \item
    $\sigma_t\in\widetilde{\FR}(S,K,\sigma,\varepsilon)$, for every $t\in[0,1]$,
    \item 
    $\sigma_0=\sigma$ in $P\times\Omega$,
    \item
    $(\sigma_t)_{|Q\times\Omega}=\sigma_{|Q\times\Omega}$, for every $t\in[0,1]$,
    \item $\sigma_1$ is holonomic in $P\times\Omega$.
\end{enumerate}
Moreover, if the $h$--principles satisfied by the realifications are $C^0$--dense, then the homotopy $\sigma_t$ can be chosen to satisfy that $\sigma^{(0)}_t$ is arbitrarily $C^0$--close to $\sigma^{(0)}$ on $P\times S$ for every $t\in[0,1]$.
\end{lemma}

\begin{proof}
Let $k=\dim(M)$ and $n=\dim_\CC(B)$. Since the complex normal bundle of $M$ is trivial, the tangent bundle of $B$ restricted to $M$ is 
$$
TB_{|M}=TM\oplus iTM\oplus (\CC^{n-k}\times M)
.$$
Choosing a Riemannian metric we can construct $W\subset\op(M)$, a totally real submanifold of dimension $n$ that contains $M$, as the image of $TM\oplus\Lambda$ under the exponential map of the tangent bundle, where $\Lambda\cong(\RR^{n-k}\times M)$ is a trivial real subbundle of the complex normal bundle of $M$.

By Remark \ref{Remark: Isomorphism between realifications and complexifications}, we can identify $\sigma_{|P\times W}$ with a formal solution of the realification $s_0\in\FR_\RR(W)$. Let $U\subset B$ be an open Stein neighbourhood of $K$ such that 
$\sigma_{|P\times\overline{U}}$ is holonomic and $C:=\overline{U}\cap W$.
Since 
$\calR_\RR(W)$ satisfies a relative to domain and fiber $h$--principle, we can find a homotopy $s_t, t\in[0,1]$ of formal solutions of $\calR_\RR(W)$ joining $s_0$ with $s_1\in\HR_\RR(W)$ such that $s_{t|Q\times W\cup P\times C}=s_{0|Q\times W\cup P\times C}$.

Again by Remark \ref{Remark: Isomorphism between realifications and complexifications}, this homotopy translates immediately to $\sigma'_t$, a homotopy of sections of $\calR_{|P\times W}$ satisfying that $\sigma'_0=\sigma_{|P\times W}$ and that $\sigma'_{t|Q\times W\cup P\times C}=\sigma_{|Q\times W\cup P\times C}$ for every $t\in[0,1]$. Since $\sigma'_t$ coincides with $\sigma$ at $P\times C$, $\sigma'_t$ extends naturally to $P\times\overline{U}$.

Now use Theorem \ref{theorem: Mergelyan} to $C^r$-approximate the family of sections $(\sigma'_t)_p:=\sigma'_t(p,\text{-})$, $(t,p)\in [0,1]\times P$ by a continuous family of holomorphic sections $(\widetilde{\sigma}_t)_p, (t,p)\in [0,1]\times P$, note that $(\widetilde{\sigma}_t)_p$ approximates $(\sigma'_t)_p$ over $S$ only, since we cannot warrant that $W$ is properly attached to $K$.

$\widetilde{\sigma}_t$ can be thougt 
as a homotopy of sections of $\calR_{|P\times\op(S)}$ that are 
$\varepsilon$--close to $\sigma$ over $P\times K$ and that coincides with $\sigma$ in $Q\times S$.
Note that since $\widetilde{\sigma}_1$ $C^r$-approximates $\sigma'_1$, $\sigma'_1$ extends $s_1$, $s_1$ is holonomic and $\calR$ is open, if the $C^r$--approximation is sufficiently close, then  the $r$--jet of $\widetilde{\sigma}_1^{(0)}$ belongs to $\widetilde{\HR}(S,K,\sigma,\varepsilon)$.

Taking a maybe closer $C^r$--approximation, we can also assume that there exists $\Omega$, a Stein neighbourhood of $S$, where the paths of sections 
$h_1(t):=(1-t)\sigma+t\widetilde{\sigma}_0$, $h_2(t):=\widetilde{\sigma}_t$ and $h_3(t):=(1-t)\widetilde{\sigma}_1+tj^r(\widetilde{\sigma}_1^{(0)}), t\in[0,1]$ are contained in $\FR_{|P\times\Omega}$. We call $h':[0,1]\to\FR_{|P\times\Omega}$ the homotopy of formal solutions of $\calR_{|P\times\Omega}$ that results after joining the paths $h_1, h_2$ and $h_3$.

The problem with $h'$ is that although $h'(t)_p:=h'(t)(p,\text{-})$ is arbitrarily close to $\sigma_p$ over $K$ for every $p\in P$, they do not necessarily have to be holonomic in $\op(K)$, so $h'(t)$ may not be in $\widetilde{\FR}(S,K,\sigma,\varepsilon)$. To fix this, consider the homotopy of germs of sections $h''(t):=(1-t)\sigma+tj^r(\widetilde{\sigma}_1^{(0)})$. Note that, taking close enough the approximation given by Theorem \ref{theorem: Mergelyan}, we can assume that $h''(t)\in\widetilde{\HR}(K)$ for every $t\in[0,1]$. Indeed $\sigma'_t=\sigma$ along $U$ for every $t\in[0,1]$, therefore $\widetilde{\sigma}_t$ approximates $\sigma$ over $U'\subset\Omega$ an open neighbourhood of $K$. Now we can interpolate $h'$ and $h''$ in $U'\setminus \op(K)$ to obtain the desired homotopy. That is, consider a smooth cutoff function $\rho$ supported in $\overline{U'}$ and such that $\rho_{|\op(K)}\equiv 1$, the desired homotopy is $\sigma_t:=h'(t(1-\rho))+h''(t\rho)$.

If the $h$--principles of the realifications are $C^0$--dense, then we can make $s_t$ to be $C^0$--close to $s_0$. Taking such an $s_t$ and replicating the argument is enough to prove the last part of the statement.
\end{proof}

The following Lemma is an analogue of Lemma \ref{Lemma: Extending holonomy through totally real submanifolds} for holomorphic germs to prove Theorems \ref{Theorem: Theorem for finite type Stein manifolds. Version that preserves holomorphy over the skeleton and closeness over proper sets.} and \ref{Theorem: Theorem for arbitrary type Stein manifolds. Version that preserves holomorphy in proper sets.}. With it we obtain holonomic sections preserving holomorphicity and pseudo-holonomy on certain sets.

\begin{lemma}\label{Lemma: Extending pseudo-holonomy through totally real submanifolds}
Let $X,\calR, B, M, K, S, Q\text{ and } P$ be like in Lemma \ref{Lemma: Extending holonomy through totally real submanifolds}.
Let $\sigma\in\mathcal{O}\widetilde{\FR}(S)$ be a holomorphic germ of formal extended solution that is already pseudo-holonomic in $P\times\op(K)\cup Q\times\op(S)$.
Then, for every $\varepsilon>0$ there exists $\Omega$ a Stein neighbourhood of $S$ and a homotopy $\sigma_t:P\times\Omega\to\calR, t\in[0,1]$ of holomorphic germs of formal extended solutions such that
\begin{enumerate}
    \item
    $\sigma_t\in\mathcal{O}\widetilde{\FR}(S,K,\sigma,\varepsilon)$ for every $t\in[0,1]$,
    \item 
    $\sigma_0=\sigma$ in $P\times\Omega$,
    \item
    $(\sigma_t)_{|Q\times\Omega}=\sigma_{|Q\times\Omega}$ for every $t\in[0,1]$,
    \item 
    $\sigma_1$ is holonomic in $P\times \Omega$.
\end{enumerate}
Moreover, if the $h$--principles satisfied by the realifications are $C^0$--dense, then the homotopy $\sigma_t$ can be chosen to satisfy that $\sigma_t^{(0)}$ is arbitrarily $C^0$-close to $\sigma^{(0)}$ on $P\times S$ for every $t\in[0,1]$.
\end{lemma}

\begin{proof}
The proof consists in just using Lemma \ref{Lemma: Extending holonomy through totally real submanifolds} to obtain a homotopy $\sigma'_t, t\in[0,1]$ that we approximate relative to $t\in\{0,1\}$ by Theorem \ref{theorem: Mergelyan} to obtain the desired homotopy $\sigma_t$.
\end{proof}

\subsection{Extending holonomy and pseudo-holonomy through adapted skeletons}\label{Subsection: Extending holonomy and pseudo-holonomy through skeletons}\hfill

The following Lemma is the result of using Lemma \ref{Lemma: Extending holonomy through totally real submanifolds} inductively over each stratum of an adapted skeleton properly attached to some Stein compact set. Therefore we obtain a version of Lemma \ref{Lemma: Extending holonomy through totally real submanifolds} where we can substitute the totally real submanifold by an adapted skeleton of the base Stein manifold.

\begin{lemma}\label{Lemma: Extending holonomy through adapted skeletons}
Let $Q\subset P$ be a compact subset in a compact 
space and let $B$ be a Stein manifold of finite type.
Let $X\to B$ be a holomorphic vector bundle equipped with a fibered open holomorphic relation $\calR\subset P\times X_\CC^{(r)}$ and assume that all its realifications satisfy a relative to domain and fiber $h$--principle. 
Let $M\subset B$ be an adapted skeleton and let $K\subset B$ be a compact subset properly attached to $M$ and let $S:=K\cup M$.
Then, for every germ of extended formal solution over $S$,  $\sigma\in\widetilde{\FR}(S)$, that is already holonomic over $P\times K\cup Q\times S$ and for every $\varepsilon>0$, there exists a homotopy $\sigma_t:P\times\op(S)\to\calR, t\in[0,1]$ of germs of formal extended solutions satisfying that
\begin{enumerate}
    \item
    $\sigma_t\in\widetilde{\FR}(S,K,\sigma,\varepsilon)$, for every $t\in[0,1]$,
    \item 
    $\sigma_0=\sigma$ in $P\times\op(S)$,
    \item
    $(\sigma_t)_{|Q\times\op(S)}=\sigma_{|Q\times\op(S)}$, for every $t\in[0,1]$,
    \item
    $\sigma_1$ is holonomic in $P\times \op(S)$.
\end{enumerate}
Moreover, if the $h$--principles satisfied by the realifications are $C^0$--dense, then the homotopy $\sigma_t$ can be chosen to satisfy that $\sigma_t^{(0)}$ is arbitrarily $C^0$--close to $\sigma^{(0)}$ on $P\times S$ for every $t\in[0,1]$.
\end{lemma}

\begin{proof}
The proof goes inductively on each stratum of the skeleton. Let $K_i:=K\cup M_{\le i}$ and $S_i:=K_i\cup M_{i+1}$ for every $i=0,\ldots,k$ where $k$ is the number of critical points of a Morse plurisubharmonic function $\phi$ defining $M$ and $M_0:=\emptyset$. After the inductive process we will have obtained finite sequences $\sigma_i$ and $h_i$, $i=1,\ldots,k$ such that
\begin{enumerate}
    \item $\sigma_i\in\widetilde{\FR}(S,K_{i-1},\sigma_{i-1},\varepsilon/(k+1))\cap\widetilde{\HR}(K_i)$.
    \item $h_i:[0,1]\to\widetilde{\FR}(S,K_{i-1},\sigma_{i-1},\varepsilon/(k+1))$ is a homotopy such that $h_i(0)=\sigma_{i-1}$ and $h_i(1)=\sigma_i$,
\end{enumerate}
where $\sigma_0=\sigma$. Putting all the homotopies $h_i$ together we obtain the homotopy $\widetilde{h}:[0,1]\to\widetilde{\FR}(S,K,\sigma,\varepsilon);t\mapsto\widetilde{\sigma}_t$ that joins $\widetilde{h}(0)=\sigma$ with $\widetilde{h}=\sigma_k\in\widetilde{\HR}(S,K,\sigma,\varepsilon)$.

Now we detail the inductive process.Taking $\sigma_0=\sigma_{-1}=\sigma$ and $h_0\equiv\sigma$ we obtain the base case. For the inductive step take $0< j<k$ and assume that we have obtained $h_i$ and $\sigma_i$ for each $i\le j$.

Since $M_{j+1}$ is a totally real disc properly attached to $K_j$ it has trivial complex normal bundle, so we can use Lemma \ref{Lemma: Extending holonomy through totally real submanifolds} on $K_j, S_j$ and $\sigma_j$ to obtain a homotopy $h:[0,1]\to\widetilde{\FR}(S_j,K_j,\sigma_j,\varepsilon/(k+1))$ defined on $\Omega$ a Stein neighbourhood of $S_j$ such that $h(1)$ is a holonomic solution of $\calR$ in $P\times\Omega$. Let $\rho:B\to [0,1]$ be a smooth cutoff function supported on $\Omega$ and such that $\rho_{|\op(S_j)}\equiv 1$. This allows us to define the sections
\begin{align*}
    h_{j+1}(t):P\times\op(S)\longrightarrow &\calR\\
    (p,x)\longmapsto    &h(t\cdot\rho(x))(p,x), t\in[0,1]
\end{align*}
obtaining the homotopy $h_{j+1}:[0,1]\to\widetilde{\FR}(S,K_j,\sigma_j,\varepsilon/(k+1))$ and the section $\sigma_{j+1}:=h_{j+1}(1)\in\widetilde{\HR}(K_{j+1})$.

In the case that the $h$--principles satisfied by the realifications are $C^0$--dense, we can take the homotopies $h_i$ to satisfy that $h_i(t)^{(0)}$, and hence $h(t)^{(0)}$, are $C^0$--close to $\sigma^{(0)}$, for every $t\in[0,1]$.
\end{proof}

The following Lemma is analogous to Lemma \ref{Lemma: Extending holonomy through adapted skeletons} for holomorphic germs. Thanks to it we can obtain holonomic sections preserving holomorphicity and pseudo-holonomy.

\begin{lemma}\label{Lemma: Extending pseudo-holonomy through adapted skeletons}
Let $X,\calR, B, M, K, S, Q\text{ and } P$ be like in Lemma \ref{Lemma: Extending holonomy through adapted skeletons}. Then, for every holomorphic germ of extended formal solution over $S$, $\sigma\in\mathcal{O}\widetilde{\FR}(S)$ that is already pseudo-holonomic over $P\times K\cup Q\times S$ and every $\varepsilon>0$, there exists a homotopy $\sigma_t:P\times\op(S)\to \calR, t\in[0,1]$ of holomorphic germs of formal extended solutions satisfying that
\begin{enumerate}
    \item 
    $\sigma_t\in\mathcal{O}\widetilde{\FR}(S,K,\sigma,\varepsilon)$, for every $t\in[0,1]$,
    \item
    $\sigma_0=\sigma$ in $P\times \op(S)$,
    \item
    $(\sigma_t)_{|Q\times\op(S)}=\sigma_{|Q\times\op(S)}$, for every $t\in[0,1]$,
    \item
    $\sigma_1$ is holonomic in $P\times\op(S)$.
\end{enumerate}
Moreover, if the $h$--principles satisfied by the realifications are $C^0$--dense, then the homotopy $\sigma_t$ can be chosen to satisfy that $\sigma_t^{(0)}$ is arbitrarily $C^0$--close to $\sigma^{(0)}$ on $P\times S$ for every $t\in[0,1]$.
\end{lemma}
\begin{proof}
The proof consist in just using Lemma \ref{Lemma: Extending holonomy through adapted skeletons} to obtain a homotopy $\sigma'_t, t\in[0,1]$ that we approximate relative to $t\in\{0,1\}$ by Corollary \ref{Corollary: Mergelyan for adapted skeletons} and Remark \ref{Remark: C^r-closeness of Mergelyan for subsets of adapted skeletons} to obtain the desired homotopy $\sigma_t$.
\end{proof}

\subsection{Proofs of the Theorems for finite type Stein manifolds.}\label{Subsection: Proofs of the  Theorems}\hfill

\begin{proof}[Proof (of Theorem \ref{Theorem:  Theorem for finite type Stein manifolds. Version that preserves holonomy in proper sets.}).]\hfill

Let $K$ be a compact set that is properly attached to $M$ and let $S:=K\cup M$ and let 
\begin{itemize}
    \item 
    $(P,Q)$ be a pair of an arbitrary compact 
    space $P$ and a compact subset $Q$,
    if the $h$--principles of the realifications are parametric and relative to parameter.
    \item
    $(P,Q)=(P,\emptyset)$ where $P$ is an arbitrary compact 
    space,
    if the $h$--principles of the realifications are just parametric.
    \item
    $(P,Q)=(P,\emptyset)$ where $P$ is just a point,
    if the $h$--principles of the realifications do not satisfy any of the previous conditions.
\end{itemize}
Now let $\sigma_0\in\widetilde{\FR_P}(S,K)$ be a germ of extended formal solution of $\calR_P$ over $S$. Now just use Lemma \ref{Lemma: Extending holonomy through adapted skeletons} to obtain a homotopy $\sigma_t\in\widetilde{\FR_P}(S,K,\sigma_0,\varepsilon)$ such that $\sigma_1\in\widetilde{\HR_P}(S)$. Note that if the $h$--principles of the realifications are $C^0$--dense then we can choose $\sigma_t$ to satisfy that $\sigma_t^{(0)}$ is arbitrarily $C^0$--close to $\sigma_0^{(0)}$.
\end{proof}

\begin{proof}[Proof (of Theorem \ref{Theorem: Theorem for finite type Stein manifolds. Version that preserves holomorphy over the skeleton and closeness over proper sets.}).]\hfill

The proof is almost the same to the one of Theorem \ref{Theorem:  Theorem for finite type Stein manifolds. Version that preserves holonomy in proper sets.} but using Lemma \ref{Lemma: Extending pseudo-holonomy through adapted skeletons} instead of \ref{Lemma: Extending holonomy through adapted skeletons}. We just also need to join the homotopy $\sigma_t$ obtained with the linear interpolation betweeen $\sigma_1$ and $j^r(\sigma_1^{(0)})$ to make the final germ to be holonomic instead of just pseudo-holonomic.
\end{proof}

\subsection{Proofs of the  Theorems for Stein manifolds of arbitrary type.}\hfill

Suppose that, for each $i=0,\ldots,k$, the section $\sigma_i$ in the proof of Lemma \ref{Lemma: Extending holonomy through adapted skeletons} is holonomic in the open Stein neighbourhood $\Omega_i=\op(M_{\le i})$. There is no reason why we can assume that $\Omega_i\subset\Omega_{i+j}$ holds. Actually, if we try to use adapted skeletons to prove analogous results for Stein manifolds of infinite type, it may happen that $\bigcap_{j=0}^\infty \Omega_{i+j}=M_{\le i}$ that has empty interior. That is the reason why, although there are adapted skeletons for any Stein manifold of arbitrary type, we do not use them to prove results analogous to the two previous Lemmas for the infinite type case.

Therefore, to prove Theorem \ref{Theorem:  Theorem for arbitrary type Stein manifolds. Version that preserves holonomy in proper sets.} we will need to obtain a sequence of germs of sections $\sigma_i$ such that they are holonomic in Stein domains $\Omega_i$ satisfying that $\Omega_i\subseteq\Omega_{i+1}$ for each $i=0,\ldots$. The price to pay for it will be to change the skeleton in each inductive step. 

The proof of Theorem \ref{Theorem:  Theorem for arbitrary type Stein manifolds. Version that preserves holonomy in proper sets.} follows similar steps to the ones given in the proof of Theorem 8.45 in \cite{EliCie} but using the $h$--principles of the realifications and the appropiate $C^r$--approximation. The construction of the isotopy $h_t$ is the same that the one that is used in the proof of Theorem 8.32 (b) of \cite{EliCie} to obtain the isotopy that is called $h^{(1)}$ in the notation of the book. 
One can also check the proof of Theorem 1.2 in \cite{Forstneric2020}.
Let us detail the process for completeness.

\begin{proof}[Proof (of Theorem \ref{Theorem:  Theorem for arbitrary type Stein manifolds. Version that preserves holonomy in proper sets.}).]\hfill

If the base manifold $B$ is of finite type then Lemma \ref{Lemma: Extending holonomy through adapted skeletons} proves the result. Let us assume that $B$ is of infinite type. 

Take a pair $(P,Q)$ such as in the proof of Theorem \ref{Theorem:  Theorem for finite type Stein manifolds. Version that preserves holonomy in proper sets.}, with abuse of notation, we will call by $\calR$ the fibered relation $P\times\calR\subset P\times X^{(r)}_\CC$. Let us also take a Morse strongly plurisubharmonic function $\phi$ and assume that it satisfies the same conditions that the one in the proof of Corollary \ref{Corollary: Mergelyan for infinite type manifolds}.

Let $\{p_i\}_{i=0}^\infty=\operatorname{Crit}(\phi)$ and let $\{c_i\}_{i=0}^\infty$ a sequence of regular values such that $c_0<\phi(p_0)<c_1<\ldots<c_i<\phi(p_i)<c_{i+1}<\ldots$ And let $\{V_i\}_{i=1}^\infty$ be like in the proof of Corollary \ref{Corollary: Mergelyan for infinite type manifolds}. 
Use Lemma \ref{Lemma: Extending holonomy through totally real submanifolds} on $M_0=\{p_0\}$ on $\sigma_0$ relative to $Q$ to obtain (after gluing with $\sigma_0$ by a cutoff function) a homotopy of formal solutions $\sigma^0_t, t\in[0,\frac{1}{2}]$ such that $\sigma^0_{\frac{1}{2}}\in\widetilde{\HR}(M_0)$. Find a regular value $c\in(\phi(p_0),c_1)$ such that the section $\sigma^0_{\frac{1}{2}}$ is holonomic in $K:=\{x\in B:\phi(x)<c\}$.

Now let us start the inductive process.
Since the stable manifold of the critical point $p_1$ for $\phi$, $M_1$, is properly attached to $K$, we can use Lemma \ref{Lemma: Extending holonomy through totally real submanifolds} again to find an open neighbourhood $U_1=\op(K\cup M_1)\subset V_1$ where there is defined a homotopy of formal solutions of $\calR_{|U_1}$, $\sigma^1_t\in\widetilde{\FR}(K\cup M_1,K,\sigma^0_\frac{1}{2},\frac{1}{2}\varepsilon), t\in[\frac{1}{2},\frac{3}{4}]$ relative to $Q$ and such that $\sigma^1_\frac{1}{2}=\sigma^0_\frac{1}{2}$ and $\sigma^1_\frac{3}{4}$ is holonomic in $U_1$. 
Extend the homotopy $\sigma^1$ to formal solution of $\calR$ in the whole manifold by $\sigma_0$ via a cutoff function (shrinking $U_1$ if necessary).

Now we reproduce the same inductive process that the one described in the proof of Corollary \ref{Corollary: Mergelyan for infinite type manifolds} but using Lemma \ref{Lemma: Extending holonomy through totally real submanifolds} instead of Theorem \ref{theorem: Mergelyan} like above.

Let $d_0=0$ and $d_k:=\sum_{i=1}^n\left(\frac{1}{2}\right)^i$ for $k\in\NN$. After the inductive process we will have obtained sequences $\phi_i, M_i, c'_i, \widetilde{V}_{i-1},V_i'$ and $U_i$ like the ones in the proof of Corollary \ref{Corollary: Mergelyan for infinite type manifolds} satisfying conditions (\ref{Contition 1 in Corollary: Mergelyan for infinite type manifolds }) to (\ref{Contition 4 in Corollary: Mergelyan for infinite type manifolds }) and a sequence of homotopies of formal solutions of $\calR$ $\sigma^i_{t}, t\in[d_i,d_{i+1}]$ such that
\begin{enumerate}
    \item 
    $\sigma^i_{d_i}=\sigma^{i-1}_{d_i}$,
    \item
    $\sigma^i_t\in\widetilde{\FR}(B,V'_{i-1},\sigma^{i-1}_{d_i},\left(\frac{1}{2}\right)^i\varepsilon)$, for every $t\in[d_i,d_{i+1}]$, 
    \item
    $\sigma^i_{d_{i+1}}$ is holonomic in $U_i$
\end{enumerate}
for each $i\in\NN$. Let $\widetilde{B}:=\bigcup_{i\in \NN}\operatorname{Int}V'_i$. Note that the sequence $\sigma^k_{d_{k+1}}\overset{k}{\to}\sigma_1$, where $\sigma_1$ is a holonomic solution of $\calR_{|\widetilde{B}}$. Therefore, the homotopy defined by $\sigma_t=\sigma^k_t$, if $t\in[d_k,d_{k+1}]$ is the one we are looking for.

The construction of $h_t$ is exactly the same as the one in the proof of Corollary \ref{Corollary: Mergelyan for infinite type manifolds}.
\end{proof}

\begin{proof}[Proof (of Theorem \ref{Theorem:  Theorem for arbitrary type Stein manifolds. Version that preserves holomorphy in proper sets.}).]\hfill

Like in the previous proofs, use the non-holomorphic version of this Theorem (Theorem \ref{Theorem:  Theorem for arbitrary type Stein manifolds. Version that preserves holonomy in proper sets.}) to obtain a homotopy of formal solutions $\sigma_t$ and an isotopy $h'_t$. Now use Corollary \ref{Corollary: Mergelyan for infinite type manifolds} for $\sigma_{t|h_1(B)}$ on $h'_1(B)$ and relatively to $\{0,1\}\times P\cup [0,1]\times P$ to obtain the isotopy $h''_t$ and the homotopy of formal solutions $\sigma_t$ that are holomorphic in $h''_1(h'_1(B))$. The desired isotopy is the result of joining $h'_t$ with $h''_t$.
\end{proof}

\subsection{Proofs of the  Theorems for open Riemann surfaces.}

\begin{proof}[Proof (of Theorem \ref{Theorem: Theorem for Riemann surfaces} and Remark \ref{Remark: Riemann surfaces results}).]
    Take a plurisubharmonic Morse exhausting function $\phi$ and let $M_\phi$ its Lagrangian skeleton. 
    Note that by Remark \ref{Remark: We have to perturb or substitute the Lagrangian skeleton} we can replicate the proofs of Theorems \ref{Theorem: Theorem for arbitrary type Stein manifolds. Version that preserves holonomy in proper sets.} and \ref{Theorem: Theorem for arbitrary type Stein manifolds. Version that preserves holomorphy in proper sets.} but assuming that all the functions $\phi_i$ coincide with $\phi$. This yields a version of Lemmas \ref{Lemma: Extending holonomy through adapted skeletons} and \ref{Lemma: Extending pseudo-holonomy through adapted skeletons} for open Riemann surfaces of arbitrary type. Now replicate the proofs of Theorems \ref{Theorem: Theorem for arbitrary type Stein manifolds. Version that preserves holonomy in proper sets.} and \ref{Theorem: Theorem for finite type Stein manifolds. Version that preserves holomorphy over the skeleton and closeness over proper sets.} using these Lemmas to obtain the desired results.
\end{proof}

\section{Applications}\label{Section: Applications.}

In this section we provide some examples of holomorphic partial differential relations in which we can apply the Theorems proven in this article to obtain some holomorphic $h$--principles over germs (or to find homotopies of formal solutions that join a formal solution with a holonomic one). The strategy will always be the same: consider a geometrical problem and express it in terms of sections of a vector bundle. Then, check if the corresponding holomorphic partial differential relation satisfies the conditions of the Theorems, in particular we will always show that the relations are thick (see Definition \ref{Definition: THR}).

We will start by reviewing the case of maps of fiberwise maximal rank, a holomorphic partial differential relation in which it is known that a version of Theorem \ref{Theorem:  Theorem for arbitrary type Stein manifolds. Version that preserves holonomy in proper sets.} is satisfied (see Theorem 1.4 in \cite{ForstericSlapar} and Theorem 8.43 and Remark 8.44 in \cite{EliCie}). Here we will rewrite this problem in our terms and give an $h$--principle for germs of maps of maximal rank. 

After that we will study the complex analogue of smooth topologically stable structures, i.e. the complexified versions of contact, even contact, Engel and symplectic structures. The case of complex contact structures is at the same situation as the one of maps of maximal rank (see Theorems 1.2 and 6.1 in \cite{Forstneric2020}) and we will proceed in the same way. The rest of the cases are completely new to the best of our knowledge.

\subsection{Maps of Fiberwise Maximal Rank: Immersions and Submersions.}

\begin{definition}
Let $B$ and $V$ be complex manifolds of dimensions $n$ and $m$ respectively. Let $f:B\to V$ be a holomorphic map whose differential map has maximal rank at every point of $B$. We will say that $f$ is a \emph{holomorphic immersion} if $n\le m$ and a \emph{holomorphic submersion} if $n\ge m$.
\end{definition}

Consider the holomorphic $1$--jet bundle $J^1_\CC(B,V)$ of the trivial bundle $B\times V\to B$. Take a point $(p,q)\in B\times V$, the fiber of the natural projection $J^1_\CC(B,V)\to B\times V$ over $(p,q)$ can be identified with the space of complex linear maps $\operatorname{Hom}_\CC(T_pB,T_qV)$. We define the holomorphic partial differential relation $\calR_{\text{Max-rank}}(B,V)$ as the set of those linear maps that have maximal rank.
Note that the differential of a holomorphic map $f:B\to V$ has maximal rank at every point of $B$ if and only if $j^1(f)\in\HR_{\text{Max-rank}}(B,V)$.

In the smooth category, the analogue of this relation is ample only in the case of smooth immersions with $\dim B<\dim V$. This is known as the Smale-Hirsch $h$--principle for immersions. Nevertheless, in the holomorphic setting, this relation is a THR (see \cite{ForstericSlapar}). Therefore, the holomorphic partial differential relation $\calR_{\text{Max-Rank}}(B,V)$ satisfies
all the hypotheses of the Theorems proven in this article. Hence we have

\begin{theorem}
Let $B$ be a Stein manifold and let $V$ be a complex manifold. Then, for every continuous family of formal solutions $\sigma_0\in P\times\calR_{\text{Max-Rank}}(B,V)$ that is holonomic over $Q\subset P$, there exists a smooth family of diffeomorphisms $h_t:B\to h_t(B)\subset B, t\in[0,1]$ and a homotopy of formal solutions $\sigma_t:P\times h_1(B)\to P\times\calR_{\text{Max-Rank}}(B,V)$ such that
\begin{enumerate}
    \item $h_0=\Id_B$,
    \item $(h_t(B), J_{|h_t(B)})$ is Stein for every $t\in[0,1]$,
    \item $\sigma^{(0)}_t$ is arbitrarily close to $\sigma^{(0)}_0$,
    \item $\sigma_t(q,\slot)=\sigma_0(q,\slot)_{|h_1(B)}$ for every $q\in Q$,
    \item $\sigma_1$ is a family of holomorphic immersions if $\dim B\le\dim V$ or submersions if $\dim B\ge\dim V$ in $h_1(B)$.
\end{enumerate}
Moreover, if $\sigma_0(p,\slot)$ is holomorphic for every $p\in P$ then, for every $t\in [0,1]$, $\sigma_t$ can be chosen to satisfy that $\sigma_t(p,\slot)$ is holomorphic for every $p\in P$.

If in addition $B$ is a Riemann surface or if it is of finite type, then $\calR_{\text{Max-Rank}}(B,V)$ satisfies a $C^0$--dense, parametric, and relative to parameter $h$--principle (holomorphic $h$--principle) for germs over any adapted skeleton that is weakly relative to 
domains properly attached to the skeleton. \qed
\end{theorem}

\subsection{Complexified topologically stable structures} \hfill

A smooth tangent distribution of rank $r$ over a smooth manifold $M$ is a 
vector subbundle of $TM$ of rank $r$.
An open set of distributions (in the Whitney topology of sections of the Grassmannian bundle) is called \emph{topologically stable} if the group $\Diff(M)$ acts transitively on it. There exist only 4 classes of topologically stable distributions (see \cite{Montgomery1} and \cite{Montgomery2}): line fields, contact distributions, even-contact distributions and Engel distributions. The same is true in the holomorphic category when we interchange the group of diffeomorphisms by the group of automorphisms (see \cite{Presas_Sola}). This transitive action of the group of diffeomorphisms/automorphisms translates into the existence of a local model (i.e. a Darboux Theorem) for each one of them. That is the reason why we have included the holomorphic twisted locally conformal symplectic structures into this subsection.

The geometric objects of this subsection will be expressed in the language of holomorphic forms with values in some line bundle $L$ over $B$. They are holomorphic sections of the vector bundles 
$\Lambda^{p}(T^*B)\otimes L, p\in\NN$, where $T^*B$ denotes the complex cotangent bundle of $B$.

Denote by $\Gamma_\text{Hol}(X)$ 
the set of holomorphic sections of the holomorphic vector bundle $X\to B$ and let $\Omega^p(B;L):=\Gamma_\text{Hol}(\Lambda^{p}(T^*B)\otimes L)$ be the space of holomorphic $p$--forms of $B$ with values in $L$.
The exterior differentiation $\Omega^{r-1}(B;L)\overset{d}{\longrightarrow}\Omega^{r}(B;L)$ can be decomposed as
\begin{equation*}
    \Omega^{r-1}(B;L)\overset{j^1_\CC}{\longrightarrow}\Gamma_\text{Hol}((\Lambda^{r-1}(T^*B)\otimes L)^{(1)}_\CC)\overset{\widetilde{D}}{\longrightarrow}\Omega^{r}(B;L).
\end{equation*}
Where $\widetilde{D}$ is induced by the symbol of the differential operator  
$d$ (that coincides with $\partial$ acting on the holomorphic forms),
$$
D:(\Lambda^{r-1}(T^*B)\otimes L)^{(1)}\to \Lambda^{r}(T^*B)\otimes L. 
$$ 
Since $D$ is an affine fibration, every holomorphic section $\sigma:B\to\Lambda^{r+1}((T^*B)\otimes L)$ can be lifted in a homotopically canonical way to a section 
$F_\sigma:B\to(\Lambda^{r}(T^*B)\otimes L)^{(1)}_\CC$ such that $D\circ F_\sigma=\sigma$. By Grauert $h$--principle (see p.6 in \cite{Gromov}) we can assume that $F_\sigma$ is holomorphic. We call such an $F_\sigma$ a \emph{formal holomorphic primitive of $\sigma$}. Note that this can be done parametrically and relative to parameter and that the $0$--jet part of a formal primitive can be chosen arbitrarily and independently of $\sigma$.

\subsubsection{Complex Contact and Even-Contact Structures.}\hfill

\begin{definition}
Let $B$ be a complex manifold of odd dimension $2n+1$ and $L\to B$ be a holomorphic line bundle. We say that a holomorphic $1$--form $\alpha$ over $B$ with values in $L$ is a \emph{complex contact form} if 
\begin{equation}\label{eq: holonomic contact condition}\tag{$\CC$-Cont}
    \alpha\wedge(d\alpha)^n\neq 0
\end{equation} 
everywhere. In this case, $\xi=\ker\alpha\subset TB$ defines a completely nonintegrable hyperplane distribution called \emph{complex contact structure over $B$}. The pair $(B,\xi)$ is a \emph{complex contact manifold}.
\end{definition}

Note that if $\alpha$ satisfies condition (\ref{eq: holonomic contact condition}), then $\alpha\wedge (d\alpha)^n$ provides a holomorphic trivialization of $K_B\otimes_\CC L^{n+1}$, where $K_B:=\Lambda^{2n+1}T^*B$ is the canonical bundle of $B$. 

Using Mori theory it is possible to classify all projective contact manifolds whose canonical bundle is not nef and that are not Fano with $b_2=1$ \cite{Projective_contact}. In the affine case, F. Forstneri\v{c} has recently proven that this holomorphic partial differential relation satisfies a version of Theorem \ref{Theorem: Theorem for arbitrary type Stein manifolds. Version that preserves holonomy in proper sets.} (see Theorem 6.1 in \cite{Forstneric2020}).

Condition (\ref{eq: holonomic contact condition}) defines a holomorphic partial differential relation in the $1$--jet of sections of the twisted holomorphic cotangent bundle  $T^*B\otimes_\CC L$
$$
\calR_{\CC\text{-Cont}}(B):=\{
a\in(T^*B\otimes L)^{(1)}_\CC|a_p^{(0)}\wedge (Da_p)^n\neq 0,\forall p\in B\}.
$$

\begin{remark}
It is worth mentioning that when the base space is a Stein manifold, the Oka principle gives that the existence of a formal solution of $\calR_{\CC\text{-Cont}}$ implies that $K_B\otimes_\CC L^{n+1}$ is holomorphically trivial.
\end{remark}

This holomorphic partial differential relation is clearly an open THR. Indeed, it intersects each fiber of $\pi^1_0:(T^*B\otimes L)_\CC^{(1)}\to T^*B\otimes L$ in the complement of a complex-analytic set of codimension 1. 
That complex-analytic set contains some principal complex subspaces and intersects the rest of them in complex codimension $1$ (see Lemma 2.1 in \cite{Forstneric2020}  for more details). 
Therefore $\calR_{\CC\text{-Cont}}(B)$ satisfies all the hypotheses of our Theorems and therefore we have the following

\begin{theorem}
Let $B$ a Stein manifold of complex dimension $2n+1$ and let $L\to B$ be a holomorphic line bundle.
Then, for every continuous family of formal solutions $\sigma_0\in P\times\calR_{\CC-\operatorname{Cont}}(B)$ that is holonomic over $Q\subset P$, there exists a smooth family of diffeomorphisms $h_t:B\to h_t(B)\subset B, t\in[0,1]$ and a homotopy of formal solutions $\sigma_t:P\times h_1(B)\to P\times\calR_{\CC-\operatorname{Cont}}(B)$ such that
\begin{enumerate}
    \item $h_0=\Id_B$,
    \item $(h_t(B), J_{|h_t(B)})$ is Stein for every $t\in[0,1]$,
    \item $\sigma^{(0)}_t$ is arbitrarily close to $\sigma^{(0)}_0$,
    \item $\sigma_t(q,\slot)=\sigma_0(q,\slot)_{|h_1(B)}$ for every $q\in Q$,
    \item $\sigma_1$ is a family of complex contact forms in $h_1(B)$.
\end{enumerate}
Moreover, if $\sigma_0(p,\slot)$ is holomorphic for every $p\in P$ then, for every $t\in[0,1]$, $\sigma_t$ can be chosen to satisfy that $\sigma_t(p,\slot)$ is holomorphic for every $p\in P$.

If in addition $B$ is a Riemann surface or if it is of finite type, then $\calR_{\CC-\operatorname{Cont}}(B)$ satisfies a $C^0$--dense, parametric, and relative to parameter $h$--principle (holomorphic $h$--principle) for germs over any adapted skeleton that is weakly relative to 
domains properly attached to the skeleton. \qed
\end{theorem}

Once again, the corresponding partial differential relation of contact forms in the smooth category is not ample. This is not the case of \emph{even-contact distributions}, the analogue of contact distributions in manifolds of even dimension. For even-contact distributions, the corresponding partial differential relation is ample even in the smooth category. That ampleness yields a full parametric, and relative to parameter and domain $h$--principle for \emph{even-contact smooth manifolds} (see \cite{McDuff_even-contact}). Therefore, it is not surprising that the complex even-contact relation also satisfies the conditions of the Theorems that we have proven. 
Nevertheless let us also review that case as it will be useful for the study of Complex Engel distributions.

\begin{definition}
Let $B$ be a complex manifold of even dimension $2n+2$ and $L\to B$ be a holomorphic line bundle. We say that a holomorphic $1$--form $\alpha$ on $B$ with values in $L$ is a \emph{complex even-contact form} if
\begin{equation}\label{eq: holonomic even-contact condition}\tag{$\CC$-ECont}
    \alpha\wedge(d\alpha)^n\neq 0
\end{equation}
everywhere. In this case, $\mathcal{E}:=\ker\alpha\subset TB$ defines a hyperplane distribution, called \emph{even-contact distribution}, such that $[\mathcal{E},\mathcal{E}]=TB$ . The pair $(B,\mathcal{E})$ is called a \emph{complex even-contact manifold}.
\end{definition}

Note that a holomorphic form $\alpha$ satisfying condition (\ref{eq: holonomic even-contact condition}) provides a nowhere vanishing section $\alpha\wedge (d\alpha)^n$ of the complex bundle of $\Lambda^{2n+1} T^*B\otimes L^{n+1}$. Since the rank of that bundle coincides with the dimension of the base space $B$, such a section will always exist if $B$ is a Stein manifold.

In the same way as in the complex contact case, condition (\ref{eq: holonomic even-contact condition}) provides a holomorphic partial differential relation in the 1-jet of sections of the holomorphic vector bundle 
$T^*B\otimes_\CC L$. Let us denote this relation by $\calR_{\CC\text{-}\operatorname{ECont}}(B)\subset(T^*B\otimes_\CC L)_\CC^{(1)}$. It is easily {checked} 
that $\calR_{\CC\text{-}\operatorname{ECont}}(B)$ is also a THR, therefore we obtain the following

\begin{theorem}\label{Theorem: h-principle for complex even-contact forms}
Let $B$ a Stein manifold of complex dimension $2n+2$ and let $L\to B$ be a holomorphic line bundle. Then, for every continuous family of formal solutions $\sigma_0\in P\times\calR_{\CC-\operatorname{ECont}}(B)$ that is holonomic over $Q\subset P$, there exists a smooth family of diffeomorphisms $h_t:B\to h_t(B)\subset B, t\in[0,1]$ and a homotopy of formal solutions $\sigma_t:P\times h_1(B)\to P\times\calR_{\CC-\operatorname{ECont}}(B)$ such that
\begin{enumerate}
    \item $h_0=\Id_B$,
    \item $(h_t(B), J_{|h_t(B)})$ is Stein for every $t\in[0,1]$,
    \item $\sigma^{(0)}_t$ is arbitrarily close to $\sigma^{(0)}_0$,
    \item $\sigma_t(q,\slot)=\sigma_0(q,\slot)_{|h_1(B)}$ for every $q\in Q$,
    \item $\sigma_1$ is a family of complex--even contact forms in $h_1(B)$.
\end{enumerate}
Moreover, if $\sigma_0(p,\slot)$ is holomorphic for every $p\in P$ then, for every $t\in [0,1]$, $\sigma_t$ can be chosen to satisfy that $\sigma_t(p,\slot)$ is holomorphic for every $p\in P$.

If in addition $B$ is a Riemann surface or if it is of finite type, then $\calR_{\CC-\operatorname{ECont}}(B)$ satisfies a $C^0$--dense, parametric, and relative to parameter $h$--principle (holomorphic $h$--principle) for germs over any adapted skeleton that is weakly relative to 
domains properly attached to the skeleton. \qed
\end{theorem}

\subsubsection{Complex Engel Structures}\hfill

\begin{definition}\label{definition: Complex Engel}
Let $B$ be a complex manifold of dimension 4 and let $L\to B$ be a holomorphic line bundle. Let $\alpha$ and $\beta$ be complex even-contact forms with values in $L$ such that
\begin{equation}\label{eq: Engel condition 1}\tag{$\CC$-Eng.1}
    \alpha\wedge\beta\wedge d\alpha\neq 0,
\end{equation}
\begin{equation}\label{eq: Engel condition 2}\tag{$\CC$-Eng.2}
    \alpha\wedge\beta\wedge d\beta= 0
\end{equation}
everywhere. The plane distribution $\mathcal D:=\ker\alpha\cap\ker\beta$ is called a \emph{complex Engel distribution} and the pair $(B,\mathcal{D})$ is a \emph{complex Engel manifold}.

We will say that a pair of formal complex even-contact structures $(a,b)$ is a formal complex Engel structure if the following conditions hold
\begin{equation}\label{eq: Formal Engel condition 1}\tag{F$\CC$-Eng.1}
    a^{(0)}\wedge b^{(0)}\wedge Da\neq 0,
\end{equation}
\begin{equation}\label{eq: Formal Engel condition 2}\tag{F$\CC$-Eng.2}
    a^{(0)}\wedge b^{(0)}\wedge Db= 0.
\end{equation}
We will denote by $\calR_{\CC\text{-Engel}}(B)\subseteq (T^*B\otimes L)^{(1)}_\CC\times_B(T^*B\otimes L)^{(1)}_\CC$ the set of formal complex Engel structures of $B$.
\end{definition}
\begin{remark}\label{Remark: Engel fibers over Even Contact}
Note that $\calR_{\CC\text{-Engel}}$ is a fibration over $\calR_{\CC\text{-ECont}}$. Indeed, let us call $E=T^*B\otimes L$ and consider the bundle $E^{(1)}_\CC\times_B E^{(1)}_\CC\overset{p_2}{\longrightarrow}E^{(1)}_\CC$, where $p_2$ is the projection to the second factor of the fibered product over $B$. Let $X$ be the restriction of the previous bundle over $\calR_{\CC\text{-ECont}}$. Now consider the bundle maps
\begin{align*}
    \Phi: X\to&(K_B\otimes L^3)\times_B \calR_{\CC\text{-ECont}}\\
    (a,b)\to&(a^{(0)}\wedge b^{(0)}\wedge Db,b)
\end{align*}
and
\begin{align*}
    \Psi: \ker\Phi\to&(K_B\otimes L^3)\times_B \calR_{\CC\text{-ECont}}\\
    (a,b)\to&(a^{(0)}\wedge b^{(0)}\wedge Da,b).
\end{align*}
Then $\calR_{\CC\text{-Engel}}:=\ker\Phi\setminus\ker\Psi$ is the complement of an affine subbundle inside another affine subbundle of $X$, and therefore a fiber bundle over $\calR_{\CC\text{-ECont}}$.
\end{remark}

Let $B$ be a Stein manifold of complex dimension 4. To find an $h$--principle for complex Engel structures take a (parametric) formal complex Engel structure $(a,b)$ on $B$ and assume for a moment that $b$ is already holonomic, i.e. $b=j^1\beta$, where $\beta$ is a complex even-contact form in $B$ with values in a line bundle $L\to B$. 
Now consider the map \begin{align*}
    \Phi_\beta:T^*B\otimes L&\to K_B\otimes L^3\\
    \alpha_x&\to \alpha_x\wedge\beta_x\wedge(d\beta)_x
\end{align*}
for each $x\in B$ and each $\alpha_x\in\pi^{-1}(x)$. Since $\beta\wedge(d\beta)\neq 0$ everywhere, $V_\beta:=\ker\Phi_\beta$ is a vector subbundle of $T^*B\otimes L$ of rank $3$ where condition (\ref{eq: Formal Engel condition 2}) is always satisfied.

Note that the Engel condition (\ref{eq: Formal Engel condition 1}) defines a THR in $(V_\beta)^{(1)}_\CC$
$$
\calR_{\CC\text{-Engel}_\beta}:=\{a\in (V_\beta)^{(1)}_\CC|a^{(0)}\wedge\beta\wedge Da\neq 0\}.
$$
Therefore we obtain the following

\begin{lemma}\label{Lemma: 2nd part of Engel h-principle}
Let $B$ be a Stein manifold of dimension 4 and let $L\to B$ be a holomorphic line bundle. 
Let $\beta$ be a parametric family of complex even-contact forms on $B$ with values in $L$ continuously dependent on $p$ in a compact CW-complex $P$. Then, for every continuous $P$--parametric family of formal solutions $\sigma_{0}\in\calR_{\CC-\text{Engel}_{\beta}}$ such that $\sigma_{0}(q,\slot)$ is holonomic for every $q$ in a compact subcomplex $Q$ of $P$, there exists a smooth family of diffeomorphisms $h_t:B\to h_t(B)\subset B, t\in[0,1]$ and a homotopy of formal solutions $\sigma_{t}:P\times h_1(B)\to\calR_{\CC-\text{Engel}_{\beta}}$ such that
\begin{enumerate}
    \item $h_0=\Id_B$,
    \item $(h_t(B),J_{|h_t(B)})$ is Stein for every $t\in [0,1]$,
    \item $\sigma_t^{(0)}$ is arbitrarily close to $\sigma_0^{(0)}$,
    \item $\sigma_t(q,\slot)=\sigma_0(q,\slot)_{|h_1(B)}$ for every $q\in Q$,
    \item $\sigma_1(p,\slot)$ is holonomic for every $p\in P$.
\end{enumerate}
Moreover, if $\sigma_0(p,\slot)$ is holomorphic for every $p\in P$, then, for every $t\in[0,1]$, $\sigma_t$ can be chosen to satisfy that $\sigma_t(p,\slot)$ is holomorphic for every $p\in P$.

If in addition $B$ is a Riemann surface or if it is of finite type, then $\calR_{\CC-\text{Engel}_{\beta}}$ satisfies a $C^0$--dense, parametric, and relative to parameter $h$--principle (holomorphic $h$--principle) for germs over any adapted skeleton that is weakly relative to 
domains properly attached to the skeleton.\qed
\end{lemma}
Due to the previous Lemma, to obtain complex Engel $h$--principles it will be enough to find a homotopy of formal complex even contact forms $a_t$ such that $(a_t,b_t)$ is a path of formal complex Engel structures, where $(a_0,b_0)=(a,b)$ and $b_t$ is the homotopy obtained in Theorem \ref{Theorem: h-principle for complex even-contact forms}. After that we just join that homotopy to the one given by Lemma \ref{Lemma: 2nd part of Engel h-principle}.

To find such a homotopy we only need to use the homotopy lifting property of the fiber bundle
$\calR_{\CC\text{-Engel}}\to\calR_{\CC\text{-ECont}}$ (see Remark \ref{Remark: Engel fibers over Even Contact}) with respect to $B\times P$ relatively to $B\times Q$ (see Proposition 4.48 in \cite{Hactcher}). In the case of the holomorphic $h$--principle we also approximate that homotopy by a holomorphic one using Theorem \ref{theorem: Mergelyan}. Thus we obtain the following

\begin{theorem}
Let $B$ be a Stein manifold of dimension 4 and let $L\to B$ be a holomorphic line bundle.
Then, for every continuous $P$--parametric family of formal complex Engel structures $(a,b)_p,p\in P$ such that $(a,b)_q$ is a pair of holonomic sections for every $q\in Q\subseteq P$, there exists a smooth family of diffeomorphisms $h_t:B\to h_t(B)\subset B, t\in[0,1]$ and a homotopy of formal complex Engel structures $(a,b)_{t}$ such that
\begin{enumerate}
    \item $h_0=\Id_B$,
    \item $(h_t(B),J_{|h_t(B)})$ is Stein for every $t\in [0,1]$,
    \item $(a,b)_0=(a,b)$,
    \item $b_t^{(0)}$ is arbitrarily close to $b^{(0)}$,
    \item $(a,b)_{t,q}=(a,b)_q$ for every $q\in Q$ and every $t\in [0,1]$,
    \item $(a,b)_1$ is a $P$--parametric family of pairs of holonomic sections.
\end{enumerate}
Moreover, if $(a,b)_p$ is a pair of holomorphic sections for every $p\in P$ then, for every $t\in[0,1]$, $(a,b)_t$ can be chosen to satisfy that $(a,b)_{t,p}$ is a pair of holomorphic sections for every $p\in P$.

If in addition $B$ is a Riemann surface or if it is of finite type, then $\calR_{\CC-\text{Engel}}$ satisfies a parametric and relative to parameter $h$--principle (holomorphic $h$--principle) for germs over any adapted skeleton. Moreover, the $h$--principle is $C^0$--dense and weakly relative to 
domains properly attached to the skeleton. 
\qed
\end{theorem}

\subsubsection{Complex twisted locally conformal symplectic structures}\hfill

A holomorphic connection $\nabla$ in a vector bundle $X\to B$, generalizes the exterior differential operator to forms with values in $X$. Indeed one can extend $\nabla:\Omega^0(B,X)=\Gamma(X)\to\Omega^1(B,X)$ to $d_\nabla:\Omega^p(B,X)\to\Omega^{p+1}(B,X)$ by
$$
d_\nabla(\theta\otimes\sigma):=d\theta\otimes\sigma+(-1)^p \theta\wedge d_\nabla\sigma
$$
for every $\theta\in\Omega^p(B)$ and every $\sigma\in\Gamma(X)$, where $d_{\nabla|\Omega^0(B,X)}=\nabla$. If the connection is flat (i.e. $d_\nabla^2=0$), the operator $d_\nabla$ defines a cohomology denoted by $H^\bullet_\nabla(B,X)$ (that is a type of Novikov cohomology when $X$ is a trivial line bundle). As for the exterior differential operator, $d_\nabla$ can be decomposed as $d_\nabla=\widetilde{D}_\nabla\circ j^1_\CC$, where $\widetilde{D}_\nabla$ is induced by $D_\nabla$, the symbol of $d_\nabla$.

\begin{definition}
Let $B$ be a complex manifold of even dimension $\dim_\CC B= 2n$ and let $L\to B$ be a holomorphic line bundle with a flat holomorphic connection $\nabla$. We say that a holomorphic $d_\nabla$--closed $2$--form $\omega$ on $B$ with values in $L$ is a \emph{complex twisted locally conformal symplectic form} if
\begin{equation}\label{eq: symplectic condition}\tag{$\CC$--TLCSymp}
    \omega^n\neq 0
\end{equation}
everywhere. In this case, $(B,L,\nabla,\omega)$ is called a \emph{complex twisted locally conformal symplectic manifold}. If $\omega=d_{\nabla}\alpha$ for some holomorphic $1$--form $\alpha$ with values in $L$ we say that $(B,L,\nabla,\omega)$ is \emph{exact}.

If the $2$--form $\omega$ is not $d_\nabla$--closed but it still satisfies condition (\ref{eq: symplectic condition}), then we will say that $\omega$ is a \emph{complex twisted locally conformal almost-symplectic form} and that $(B,L,\nabla,\omega)$ is a \emph{complex twisted locally conformal almost-symplectic manifold}.
\end{definition}

\begin{remark}\label{Remark: Relation between TLCSymp, LCSymp and Symp}
   Let us provide some observations
   \begin{enumerate}
       \item 
       If $L$ is trivial and $\nabla$ is the exterior differential operator $d$, we recover the classical concepts of complex symplectic manifolds, so in this case we can omit the words ``\emph{twisted locally conformal}''.
       \item 
       If $L$ is trivial and $\nabla=d+\mu$, $\mu\in\Omega^1(B,L)$,  $\omega$ is locally expressed as a multiple of a symplectic form. Indeed, taking local integrals $\mu_{|U}=df$ we have that
       $$
       d(e^f\omega)=e^f(d\omega+\mu\wedge\omega)=e^f(d_\nabla\omega)=0,
       $$
       so the 2--forms $e^f\omega$ are symplectic in $U$. Hence, we recover the usual setting for complex locally conformal symplectic manifolds.
       \item If $L$ is not trivial one cannot choose a global 1-form $\mu$ as before. But that can be done locally, therefore $\omega$ can also be locally expressed as a multiple of a symplectic form. Nevertheless, the terms \emph{locally conformal symplectic manifolds} are often used just for trivial line bundles in the literature. We have added the term \emph{twisted} regarding to the nontriviality of $L$ as it has been previously done in \cite{Conformal_Symplectic}.
       \end{enumerate}
\end{remark}

There are only a few known examples of compact hyperkähler manifolds (i.e. complex conformal symplectic compact Kähler manifolds in which $L$ is trivial). They are Hilbert schemes of $K3$--manifolds and generalized Kummer manifolds \cite{hyperkahler_beaubille}, and also another singular example in real dimension 20 constructed by G. O'Grady \cite{Hyperkahler_new_example}. The ``twisted case'' is not in a very different situation since every complex twisted locally conformal-symplectic compact Kähler manifold is a cyclic cover of a hyperkähler manifold \cite{Conformal_Symplectic}. Here we discuss the situation in which the manifold is not compact but Stein, providing a source of many examples of complex twisted locally conformal symplectic manifolds.

Similarly to the complex contact case, if $\omega$ satisfies condition (\ref{eq: symplectic condition}), then $\omega^n$ is a trivialization of $K_B\otimes L^n$. It also occurs that the condition of being exact symplectic defines an open THR, $\calR^0_{\CC\text{-TLCSymp}}(B,L,\nabla)\subset(T^*B\otimes L)^{(1)}_\CC$.

More generally, let $a\in H_\nabla^2(B,L)$ and let $\eta$ be a closed holomorphic $2$--form representing $a$. We define the holomorphic partial differential relation $\calR^\eta_{\CC\text{-TLCSymp}}(B,L,\nabla)\subset (T^*B\otimes L)^{(1)}_\CC$ as the set of sections $\sigma$ such that $(D_\nabla\sigma+\eta)^n\neq 0$. The relation $\calR^\eta_{\CC\text{-TLCSymp}}(B,L,\nabla)$ is a THR, so therefore we have the following

\begin{lemma}\label{Lemma: h-principle of twisted locally conformal symplectic formal primitives.}
Let $B$ a Stein manifold of complex dimension $2n$ and let $L\to B$ be a holomorphic line bundle with a flat holomorphic connection $\nabla$.
Let $a\in H_\nabla^2(B,L)$ be a fixed cohomology class and let $\eta$ be a holomorphic closed $2$--form representing that class. Then, for every continuous family of formal solutions $\sigma_0\in P\times\calR^\eta_{\CC\text{-TLCSymp}}(B,L,\nabla)$ that is holonomic over $Q\subset P$, there exists a smooth family of diffeomorphisms $h_t:B\to h_t(B)\subset B, t\in[0,1]$ and a homotopy of formal solutions $\sigma_t:P\times h_1(B)\to P\times\calR^\eta_{\CC\text{-TLCSymp}}(B,L,\nabla)$ such that
\begin{enumerate}
    \item $h_0=\Id_B$,
    \item $(h_t(B), J_{|h_t(B)})$ is Stein for every $t\in[0,1]$,
    \item $\sigma^{(0)}_t$ is arbitrarily close to $\sigma^{(0)}_0$,
    \item $\sigma_t(q,\slot)=\sigma_0(q,\slot)_{|h_1(B)}$ for every $q\in Q$,
    \item $\sigma_1$ is holonomic.
\end{enumerate}
Moreover, if $\sigma_0(p,\slot)$ is holomorphic for every $p\in P$ then, for every $t\in [0,1]$, $\sigma_t$ can be chosen to satisfy that $\sigma_t(p,\slot)$ is holomorphic for every $p\in P$.

If in addition $B$ is a Riemann surface or if it is of finite type, then $\calR^\eta_{\CC\text{-TLCSymp}}(B,L,\nabla)$ satisfies a $C^0$--dense, parametric and relative to parameter $h$--principle (holomorphic $h$--principle) for germs over any adapted skeleton that is weakly relative to 
domains properly attached to the skeleton. \qed
\end{lemma}

The previous Lemma will be useful to prove some $h$--principles over germs of twisted locally conformal symplectic structures. Here twisted locally conformal almost-symplectic forms play the role of formal solutions.

\begin{definition} 
Let $a\in H_\nabla^2(B,L)$ be a fixed cohomology class and let $S\subset B$ be a closed subset of $B$. We denote the set of germs of twisted locally conformal symplectic forms over $S$ representing $a$ by $\mathbf{TLCS}^a(S)$. We denote the set of germs of twisted locally conformal almost-symplectic forms over $S$ by $\mathbf{TLCAS}(S)$.
\end{definition}

\begin{theorem}\label{Theorem: TLCSymp h-principle}
Let $B$ be a Stein manifold of complex dimension $2n$ and let $L\to B$ be a holomorphic line bundle with a flat holomorphic connection $\nabla$.
Let $a\in H_\nabla^2(B,L)$ be a fixed cohomology class and let $\eta$ be a holomorphic closed $2$--form representing that class. Then, for any continuous family of complex twisted locally conformal almost-symplectic forms $\omega_{p,0}, p\in P$ such that $\omega_{q,0}$ represents $a$ for every $q\in Q\subset P$, there exists a smooth family of diffeomorphisms $h_t:B\to h_t(B)\subset B, t\in[0,1]$ and a homotopy of $P$--parametric twisted locally conformal almost-symplectic forms $\omega_{p,t}$ such that
\begin{enumerate}
    \item $h_0=\Id_B$,
    \item $(h_t(B),J_{|h_t(B)})$ is Stein for every $t\in[0,1]$,
    \item $\omega_{q,t}=\omega_{q,0|h_1(B)}$ for every $q\in Q$,
    \item $\omega_{p,1}$ represents $a$ for every $p\in P$.
\end{enumerate}
In particular, $\omega_{p,1}$ is a twisted locally conformal symplectic form for every $p\in P$. 

If in addition $B$ is a Riemann surface or if it is of finite type, then the inclusion $\mathbf{TLCS}^a(M)\overset{i}{\hookrightarrow}\mathbf{TLCAS}(M)$ is a weak homotopy equivalence for every adapted skeleton $M\subset B$.
\end{theorem}
\begin{proof}
Let $\sigma_{p,0}$ be a continuous family of holomorphic formal primitives of the $2$--forms $\omega_{p,0}-\eta$ such that $\sigma_{q,0}$ is holonomic for every $q\in Q$. Apply Lemma \ref{Lemma: h-principle of twisted locally conformal symplectic formal primitives.} to obtain the desired $h_t$ and a homotopy of holomorphic formal solutions of $\calR^\eta_{C-\text{TLCSymp}}(B,L,\nabla)$, $\sigma_{p,t}$. The homotopy of $P$--parametric twisted locally conformal almost-symplectic forms $\omega_{p,t}:=\widetilde{D}_\nabla\sigma_{p,t}+\eta$ satisfies the desired properties.

Now assume that $B$ is a Stein manifold of finite type and let $M$ be an adapted skeleton of $B$. To prove the weak homotopy equivalence in the statement, it suffices to consider $\SS^k$--families of twisted locally conformal almost-symplectic forms for the surjectivity of $i_*:\pi_k(\mathbf{TLCS}^a(M))\to\pi_k(\mathbf{TLCAS}(M))$,  and $\SS^k\times[0,1]$--families relative to $\SS^k\times\{0,1\}$ for the injectivity.
\end{proof}
\begin{remark}
    Note that by Remark \ref{Remark: Relation between TLCSymp, LCSymp and Symp} we have the corresponding analogues of Theorem \ref{Theorem: TLCSymp h-principle} for complex locally conformal symplectic structures when $L$ is trivial and for complex symplectic structures when $\nabla=d$.
\end{remark}

\newpage
\nocite{*}
\bibliographystyle{alpha}
\bibliography{main.bib}
\end{document}